\pdfoutput=1
\documentclass[12pt,centertags]{amsart}

\usepackage{fullpage}
\usepackage{latexsym}
\usepackage[english]{babel}
\usepackage [autostyle, english = american]{csquotes}
\MakeOuterQuote{"}

\usepackage{amsmath}
\usepackage{amsthm}
\usepackage{amssymb}
\usepackage{graphicx}
\usepackage{amsfonts}
\usepackage{hyperref}
\usepackage{bbm}
\usepackage{mathrsfs}
\usepackage{accents}
\usepackage{mathtools}	

\usepackage{color}
\usepackage[usenames,dvipsnames]{xcolor}

\usepackage{mparhack}	
\usepackage{comment}	
\usepackage{marvosym}	
\usepackage{enumitem}	
\usepackage{array}		    
\usepackage{todonotes}

\numberwithin{equation}{section}
\theoremstyle{plain}
\newtheorem{Prop}{Proposition}[section]
\newtheorem{Thm}[Prop]{Theorem}
\newtheorem*{Thm*}{Theorem}
\newtheorem{Lem}[Prop]{Lemma}
\newtheorem{Cor}[Prop]{Corollary}

\theoremstyle{definition}
\newtheorem{Def}[Prop]{Definition}

\theoremstyle{remark}
\newtheorem{Rem}[Prop]{Remark}

\def\vint_#1{\mathchoice%
          {\mathop{\kern 0.2em\vrule width 0.6em height 0.69678ex
depth -0.58065ex
                  \kern -0.8em \intop}\nolimits_{\kern -0.4em#1}}%
          {\mathop{\kern 0.1em\vrule width 0.5em height 0.69678ex
depth -0.60387ex
                  \kern -0.6em \intop}\nolimits_{#1}}%
          {\mathop{\kern 0.1em\vrule width 0.5em height 0.69678ex
              depth -0.60387ex
                  \kern -0.6em \intop}\nolimits_{#1}}%
          {\mathop{\kern 0.1em\vrule width 0.5em height 0.69678ex
depth -0.60387ex
                  \kern -0.6em \intop}\nolimits_{#1}}}
\def\vintslides_#1{\mathchoice%
          {\mathop{\kern 0.1em\vrule width 0.5em height 0.697ex depth -0.581ex
                  \kern -0.6em \intop}\nolimits_{\kern -0.4em#1}}%
          {\mathop{\kern 0.1em\vrule width 0.3em height 0.697ex depth -0.604ex
                  \kern -0.4em \intop}\nolimits_{#1}}%
          {\mathop{\kern 0.1em\vrule width 0.3em height 0.697ex depth -0.604ex
                  \kern -0.4em \intop}\nolimits_{#1}}%
          {\mathop{\kern 0.1em\vrule width 0.3em height 0.697ex depth -0.604ex
                  \kern -0.4em \intop}\nolimits_{#1}}}

\newcommand{\intav}{\vint}
\newcommand{\aveint}[2]{\mathchoice
          {\mathop{\kern 0.2em\vrule width 0.6em height 0.69678ex
depth -0.58065ex
                  \kern -0.8em \intop}\nolimits_{\kern -0.45em#1}^{#2}}%
          {\mathop{\kern 0.1em\vrule width 0.5em height 0.69678ex
depth -0.60387ex
                  \kern -0.6em \intop}\nolimits_{#1}^{#2}}%
          {\mathop{\kern 0.1em\vrule width 0.5em height 0.69678ex
depth -0.60387ex
                  \kern -0.6em \intop}\nolimits_{#1}^{#2}}%
          {\mathop{\kern 0.1em\vrule width 0.5em height 0.69678ex
depth -0.60387ex
                  \kern -0.6em \intop}\nolimits_{#1}^{#2}}}

\DeclareMathOperator{\spn}{span}
\DeclareMathOperator{\supp}{supp}
\DeclareMathOperator{\diam}{diam}
\DeclareMathOperator{\dv}{div}
\DeclareMathOperator{\Tr}{Tr}

\newcommand{\set}[2]{\left\{#1 : #2\right\}}
\newcommand{\emp}{\emptyset}
\newcommand{\R}{\mathbb{R}}

\newcommand{\del}{\partial}

\newcommand{\eps}{\varepsilon}
\newcommand{\abs}[1]{\left| #1 \right|}
\newcommand{\norm}[1]{\left\| #1 \right\|}
\newcommand{\inv}[1]{{#1}^{-1}}
\newcommand{\dx}{\, dx}
\newcommand{\loc}{\text{\rm loc}}
\newcommand{\Om}{\Omega}
\newcommand{\om}{\omega}
\newcommand{\inp}[2]{\big\langle #1,#2\big\rangle}
\newcommand{\gr}{\nabla}
\newcommand{\wto}{\rightharpoonup}
\newcommand{\h}{\mathcal{H}}

\newcommand{\hh}{\mathbb{H}}
\newcommand{\X}{\mathfrak{X} }
\newcommand{\Xu}{\X u}
\newcommand{\XX}{\X\X}
\newcommand{\dvh}{\dv_{H}}
\newcommand{\A}{\mathcal{A}}
\newcommand{\Aeps}{\mathcal{A\,_\eps}}
\newcommand{\F}{\textsc{F}}
\newcommand{\weight}{\F\left(|\X u|\right)}
\newcommand{\hw}[2]{HW^{#1,#2}}
\newcommand{\dhh}{d_{\hh^n}}
\newcommand{\normh}[1]{\|#1\|_{\hh^n}}

\title[On local Lipschitz regularity for Quasilinear equations in the Heisenberg Group]
{On local Lipschitz regularity for Quasilinear equations in the Heisenberg Group}

\author[Shirsho Mukherjee]{Shirsho Mukherjee}
\address[S.\ Mukherjee]{Department of Mathematics, 
Johns Hopkins University, 404 Krieger Hall, 3400 N. Charles Street, Baltimore MD 21218, USA.}
\email{smukhe20@jhu.edu}

\begin{document}

\begin{abstract}
The goal of this article is to establish local Lipschitz continuity of solutions for 
a class of degenerated sub-elliptic equations of divergence form, in the 
Heisenberg Group.  The considered hypothesis for the growth and ellipticity condition, is a natural generalisation of the sub-elliptic $p$-Laplace equation and more general 
quasilinear equations with polynomial or exponential type growth.
\end{abstract}

\date{\today}
\maketitle
\setcounter{tocdepth}{2}
\phantomsection
\addcontentsline{toc}{section}{Contents}

\section{Introduction}\label{sec:Introduction}
 Lipschitz continuity of weak solutions for variational problems in the Heisenberg Group $\hh^n$, has been studied in \cite{Zhong}, where equations with growth conditions of $p$-Laplacian type was considered. 
The purpose of this paper is to reproduce the result, for a larger class of more general quasilinear equations.

In a domain $\Om\subset \hh^n$, for $n\geq 1$, we consider the equation 
\begin{equation}\label{eq:eq}
 \dvh \A(\X u)=\sum_{i=1}^{2n} X_i(\A_i(\X u))= 0, 
\end{equation}
where $X_1,\ldots,X_{2n}$ are the horizontal vector fields, 
$ \X u = (X_1u,\ldots,X_{2n}u)$ is the horizontal gradient of a function $ u :\Om \to \R$, the horizontal divergence $\dvh$ defined similarly and 
$\A = (\A_1,\A_2,\ldots,\A_{2n}):\R^{2n} \to \R^{2n}$ for given data
$\A_i \in C^1(\R^{2n})$. We denote  $D\A(z)$ as the $2n\times 2n$ Jacobian matrix $ (\del \A_i(z)/\del z_j)_{ij}$ for $z \in \R^{2n}$. In addition, we assume that $D\A(z)$ is symmetric and 
satisfies
\begin{equation}\label{eq:str cond}
\begin{aligned}
 \frac{g(|z|)}{|z|}\,|\xi|^2\, \leq \,\,&\inp{D\A(z)\,\xi}{\xi}
 \,\leq\, L\,\frac{g(|z|)}{|z|}\,|\xi|^2;\\
 &|\A(z)|\,\leq\, L\, g(|z|),
 \end{aligned}
\end{equation}
for every $ z,\xi \in \R^{2n}$, where $L\geq 1$ and $g:[0,\infty)\to [0,\infty)$ such that $g\in C^1((0,\infty)),\, g(0) = 0$ and there exists constants $ g_0 \geq \delta > 0$, such that
\begin{equation}\label{eq:g prop}
 \delta   \leq \frac{tg'(t)}{g(t)} \leq   g_0, \quad\text{for all}\ t>0.
\end{equation}

In the Euclidean setting, conditions \eqref{eq:str cond} and \eqref{eq:g prop} have been introduced by Lieberman \cite{Lieb--gen}, in order to produce a natural extension of the 
structure conditions for elliptic operators in divergence form previously considered in Ladyzhenskaya-Ural'tseva \cite{Lady-Ural}, which in his words is "in a sense, the best generalization". We refer to \cite{Simon,Dib,Tolk,Gia-Giu--min,Gia-Giu--div,Uhlen,Evans,Lewis} and references therein, for 
earlier works on regularity theory of elliptic equations in divergence form, 
including the 
$p$-Laplace equations in the setting of the Euclidean spaces. 

A prominent special case appears from minimization of the scalar variational integral
$$ I(u) = \int_\Om G(|\X u|)\dx ,$$
where $G(t) = \int^t_0 g(s)\,ds$ with $g$ satisfying \eqref{eq:g prop}. The corresponding Euler-Lagrange equation 
\begin{equation}\label{eq:minprob}
\dvh \Big( g(|\X u|)\,\frac{\X u}{|\X u|}\Big)= 
\sum_{i=1}^{2n} X_i\Big( g(|\X u|)\,\frac{X_i u}{|\X u|}\Big) =0, 
\end{equation}
forms a prototype example of the equation 
\eqref{eq:eq} with $\A(z)=zg(|z|)/|z|$; in this case, by explicit computation and 
the condition \eqref{eq:g prop}, 
it is not difficult to show that 
$$ \min\{1,\delta\}\frac{g(|z|)}{|z|}\,|\xi|^2\, \leq \,\inp{D\A(z)\,\xi}{\xi}
 \,\leq\, \max\{1,g_0\}\frac{g(|z|)}{|z|}\,|\xi|^2$$
for every $z,\xi\in\R^{2n}$, 
which resembles \eqref{eq:str cond}. In particular, if $g(t) = t^{p-1}$ for 
$1<p<\infty$, then $g$ satisfies \eqref{eq:g prop} with $\delta=p-1=g_0$ and 
\eqref{eq:minprob} becomes the sub-elliptic $p$-Laplace equation 
$ \dvh (|\X u|^{p-2}\X u) = 0$.
The condition  
\eqref{eq:g prop} can appear naturally if one considers defining
\begin{equation}\label{eq:defg0}
 \delta \,= \,\inf_{t>0}\ \frac{tg'(t)}{g(t)} \quad \text{and}\quad 
g_0\, =\, \sup_{t>0}\ \frac{tg'(t)}{g(t)}.
\end{equation}
However, positivity and finiteness of the constants in \eqref{eq:defg0}, are  
essential and the techniques of this paper do not apply to the borderline cases e.g. 
$\delta=0$. Hence, more singular 
equations like $\dvh(\X u/|\X u|)=0$ or $\dvh(\X u/\sqrt{1+|\X u|^2})=0$ are 
excluded from our setting. 
 
The conditions \eqref{eq:str cond} and \eqref{eq:g prop} encompass 
 quasilinear equations for a large class of structure function $g$. Some 
 natural examples include functions having growth similar to that of power-like functions and 
there logarithmic perturbations. We enlist two particular examples:  
\begin{align*}
&(1) \ \ g(t) = (e+t)^{a+b\sin(\log\log(e+t))}-e^a\qquad \text{for}\ b>0, a\geq 1+b\sqrt{2}\\
&(2) \ \ g(t) = t^\alpha(\log(a+t))^\beta \qquad\qquad\qquad\quad\, \text{for}
\ \alpha,\beta>0, a\geq 1, 
\end{align*}
see \cite{Fusco-Sb, Mar2}. 
In addition, multiple candidates satisfying condition 
\eqref{eq:g prop} can be glued together to form the function $g$. 
A suitable gluing of the monomials 
$t^{\alpha-\eps}, t^\alpha$ and $t^{\beta+\eps}$ for $\beta>\alpha>\eps$  as 
shown in \cite{Lieb--gen}, can be constructed in such a way that certain 
non-standard growth conditions (so called $(p,q)$-growth condition) of 
Marcellini \cite{Mar1}, can also be analyzed. 

Regularity theory in the sub-elliptic setting goes back to the seminal work 
of H\"ormander \cite{Hor} in 1967, from which one can verify that sub-elliptic linear operators are hypoelliptic and hence, distributional solutions of sub-Laplace equation are smooth. 
Since then, regularity of quasilinear equations in the sub-elliptic setting, has been a subject of extensive investigation throughout the following decades. We refer to 
\cite{Cap--reg,C-D-G,Cap-Garo,Foglein,Dom-Man--reg,Dom-Man--cordes,Marchi,Man-Min,Dom} etc. for earlier results. Local Lipschitz continuity of weak solutions for $p$-Laplace equation in $\hh^n$, 
has been recently proved in \cite{Zhong}. The techniques used in there, paves the 
way for this paper.

The natural domain for the weak solution of \eqref{eq:eq} is the 
Horizontal Orlicz-Sobolev space $HW^{1,G}(\Om)$ (see Section \ref{sec:Preliminaries} for details), defined similarly as the Horizontal Sobolev space 
$HW^{1,p}(\Om)$ as in \cite{Man-Min, Min-Z-Zhong, Zhong}. The natural class of metrics 
for localizing the estimates, is the class of homogeneous metrics equivalent to the CC-metric (see Section \ref{sec:Preliminaries}). 

The following theorem is the main result of this paper. 
\begin{Thm}\label{thm:mainthm}
 If $ u \in \hw{1}{G}(\Om)$ is a weak solution of equation \eqref{eq:eq}
 equipped with the structure condition \eqref{eq:str cond}, where 
 $G(t) = \int^t_0 g(s)\,ds$ and  
 $g$ satisfies \eqref{eq:g prop} for some $g_0\geq \delta>0$, then 
 $\X u \in L^\infty_\loc(\Om, \R^{2n})$. Moreover, for any CC-metric ball $B_r \subset \Om $, we have the estimate 
 \begin{equation}\label{eq:locbound}
  \sup_{B_{\sigma r}}\ G(|\X u|)\leq \frac{c}{(1-\sigma)^Q}\intav_{B_r}G(|\X u|)\dx 
 \end{equation}
 for any $0<\sigma<1$, where $c = c(n,\delta,g_0,L) > 0 $ is a constant. 
\end{Thm}
Although the above theorem is stated in terms of CC-metric balls, but as evident from its proof, \eqref{eq:locbound} holds for any homogeneous metrics of 
$\hh^n$ with appropriate choice of constant $c$. 
 
This paper is organised as follows. We provide 
some essential preliminaries on Heisenberg group, Orlicz-Sobolev spaces and sub-elliptic equations in 
Section \ref{sec:Preliminaries}. Then we prove several Caccioppoli type 
inequalities of the horizontal and vertical derivatives in Section \ref{sec:LocalB},   
followed by the proof of Theorem \ref{thm:mainthm} in the end.

Lastly, we remark that local $C^{1,\alpha}$-regularity of weak solutions of the 
$p$-Laplace equation in $\hh^n$, has been shown recently in \cite{Muk-Zhong}; 
the techniques can be adopted to show the same result for the equation \eqref{eq:eq} with structure conditions \eqref{eq:str cond} and \eqref{eq:g prop}, as well. Furthermore, $C^{1,\alpha}$-regularity can also be shown for general quasilinear equations of the form 
\begin{equation}\label{eq:AB}
\dvh A(x,u,\X u)+ B(x,u,\X u) =  0, 
\end{equation}
with appropriate growth and ellipticity conditions. Towards this pursuit, 
the estimate \eqref{eq:locbound} is necessary for both equations \eqref{eq:eq} and \eqref{eq:AB} and to obtain uniform $C^{1,\alpha}$ estimates, further technical difficulties appear that require lengthier and more delicate analysis.  Therefore, they are omitted from this paper and 
have been addressed in a follow up article \cite{Muk}.

\subsection*{Acknowledgement}
The author was supported by European Union’s Seventh Framework Program  "Metric Analysis For Emergent Technologies (MAnET)", a Marie Curie Actions-Initial Training Network, under Grant Agreement No. 607643. 
The author is thankful to Xiao Zhong for valuable suggestions and comments.

\section{Preliminaries}\label{sec:Preliminaries}
In this section, we fix the notations used and introduce the  
Heisenberg Group $\hh^n$. Also, we provide some essential facts on 
Orlicz-Sobolev spaces and sub-elliptic equations.

Throughout this paper, we shall denote a postive constant by $c$ which may vary from 
line to line. But $c$ would depend only on $n$, the constant $g_0$ of \eqref{eq:g prop} and and $L$ of \eqref{eq:str cond}, unless it is explicitly specified 
otherwise. The dependence on $\delta$ of \eqref{eq:g prop} shall appear at 
the very end. 
\subsection{Heisenberg Group}\label{subsec:Heisenberg Group}\noindent
\\
Here we provide the definition and properties of Heisenberg group  
that would be useful in this paper.  
For more details, we 
refer the reader to the books \cite{Bonfig-Lanco-Ugu, C-D-S-T}. 
\begin{Def}\label{def:H group}
 For $n\geq 1$, the \textit{Heisenberg Group} denoted by $\hh^n$, is identified to  the Euclidean space 
$\R^{2n+1}$ with the group operation 
\begin{equation}\label{eq:group op}
 x\cdot y\, := \Big(x_1+y_1,\ \dots,\ x_{2n}+y_{2n},\ t+s+\frac{1}{2}
\sum_{i=1}^n (x_iy_{n+i}-x_{n+i}y_i)\Big)
\end{equation}
for every $x=(x_1,\ldots,x_{2n},t),\, y=(y_1,\ldots,y_{2n},s)\in {\mathbb H}^n$.
\end{Def}
Thus, $\hh^n$ with the group operation \eqref{eq:group op} forms a 
non-Abelian Lie group, whose left invariant vector fields corresponding to the
canonical basis of the Lie algebra, are
\begin{equation}\label{eq:Xdef}
X_i=  \partial_{x_i}-\frac{x_{n+i}}{2}\partial_t, \quad
X_{n+i}=  \partial_{x_{n+i}}+\frac{x_i}{2}\partial_t,
\end{equation}
for every $1\leq i\leq n$ and the only
non zero commutator $ T= \partial_t$. 
We have 
\begin{equation}\label{eq:comm}
  [X_i\,,X_{n+i}]=  T\quad 
  \text{and}\quad [X_i\,,X_{j}] = 0\ \ \forall\ j\neq n+i.
\end{equation}
We call $X_1,\ldots, X_{2n}$ as \textit{horizontal
vector fields} and $T$ as the \textit{vertical vector field}. 
For a scalar function $ f: \hh^n \to \R$, we denote
$$  \X f  := (X_1f,\ldots, X_{2n}f)\quad \text{and}\quad 
\XX f :=  (X_jX_i f)_{i,j} $$
as the \textit{Horizontal gradient} and \textit{Horizontal Hessian}, 
respectively. 
From \eqref{eq:comm}, we have the following
trivial but nevertheless, an important inequality 
\begin{equation}\label{eq:T-D}
|Tf|\leq 2|\XX f| . 
\end{equation}
For a vector valued function 
$F = (f_1,\ldots,f_{2n}) : \hh^n\to \R^{2n}$, the 
\textit{Horizontal divergence} is defined as 
$ \dvh (F)  :=  \sum_{i=1}^{2n} X_i f_i$. 

The Euclidean gradient of a 
function $h: \R^{k} \to \R$, shall be denoted by
$\gr h=(D_1h,\ldots,D_{k} h)$ with $D_j = \del_{x_j}$ and the Hessian matrix by 
$D^2h=(D_iD_j h)_{i,j}$.

A piecewise smooth rectifiable curve $\gamma $ is called a \textit{horizontal curve} if its tangent vectors are contained in the \textit{horizontal sub-bundle} 
$\h = \spn\{X_1,\ldots,X_{2n}\}$, that is  
$\gamma'(t) \in \h_{\gamma(t)}$ for almost every $t$. 
For any 
$x, y \in \hh^n$, if the set of all horizontal curves is denoted as  
$$\Gamma(x,y) = \set{\gamma:[0,1]\to \hh^n}{\gamma(0) = x,\gamma(1) = y, 
\ \gamma'(t)\in\h_{\gamma(t)}},$$  
then Chow's accessibility theorem (see \cite{Chow}) gurantees $\Gamma(x,y) \neq \emp $. 
The \textit{Carnot-Carath\`eodory metric} (CC-metric) is defined in terms of the length $\ell(\gamma)$ of 
horizontal curves, as  
\begin{equation}\label{eq:cc metric}
d(x,y)= \ \inf\set{\ell(\gamma)}{\gamma \in \Gamma(x,y)}. 
\end{equation}
This is equivalent to the \textit{Kor\`anyi metric} 
$ \dhh(x,y)= \normh{y^{-1}\cdot x}$,  
where $\normh{\cdot}$ is the Kor\`anyi norm, see \cite{C-D-S-T}. We shall require  
the following norm, which is equivalent to the Kor\`anyi norm, 
\begin{equation}\label{eq:norm}
 \|x\| :=  \Big(\sum_{i=1}^{2n} x_i^2+ |t|\Big)^\frac{1}{2} 
 \quad\text{for all}\ \, x=(x_1,\ldots,x_{2n}, t)\in \hh^n. 
\end{equation}
Throughout this article we use CC-metric balls denoted by $B_r(x) = \set{y\in\hh^n}{d(x,y)<r}$ for $r>0$ and $ x \in \hh^n $. However, by virtue of the equivalence 
of the metrics, all assertions for CC-balls can be restated to Kor\`anyi balls 
or metric balls defined by the norm \eqref{eq:norm}. 
 
Throughout this paper, the Hausdorff dimension with respect to $d$ shall be denoted as 
\begin{equation}\label{eq:dimh}
Q := 2n+2,
\end{equation}
which is also the homogeneous 
dimension of the group $\mathbb H^n$. The Lebesgue measure of $\R^{2n+1}$ is a Haar measure of $\hh^n$; hence for any metric ball $B_r$, we have that 
$|B_r| = c(n)r^Q$. 

For $ 1\leq p <\infty$, the \textit{Horizontal Sobolev space} $HW^{1,p}(\Omega)$ consists
of functions $u\in L^p(\Omega)$ such that the distributional horizontal gradient $\X u$ is in $L^p(\Omega\,,\R^{2n})$.
$HW^{1,p}(\Omega)$ is a Banach space with respect to the norm
\begin{equation}\label{eq:sob norm}
  \| u\|_{HW^{1,p}(\Omega)}= \ \| u\|_{L^p(\Omega)}+\| \X u\|_{L^p(\Omega,\R^{2n})}.
\end{equation}
We define $HW^{1,p}_{\loc}(\Omega)$ as its local variant and 
$HW^{1,p}_0(\Omega)$ as the closure of $C^\infty_0(\Omega)$ in 
$HW^{1,p}(\Omega)$ with respect to the norm in \eqref{eq:sob norm}. 
The Sobolev Inequality has the following version in the sub-elliptic setting (see \cite{C-D-G, C-D-S-T}).
\begin{Thm}[Sobolev Inequality]\label{thm:sob emb}
Let $B_r\subset {\mathbb H}^n$ and $1<q<Q$. There exists a constant $c=c(n,q)>0$ such that for all $u \in HW^{1,q}_0(B_r)$, we have 
\begin{equation}\label{eq:sob emb}
\left(\intav_{B_r}| u|^{\frac{Q q}{Q-q}}\, dx\right)^{\frac{Q-q}{Q q}}
\leq\, cr \left(\intav_{B_r}| \X u|^q\, dx\right)^{\frac 1 q}.
\end{equation}
\end{Thm}
We remark that the Lipschitz continuity that is considered, is implied in the sense 
of Folland-Stein \cite{Folland-Stein--book}, i.e. the Lipschitz continuity with respect to the CC-metric. It does 
not make any assertion on the regularity of the vertical derivative.
\subsection{Orlicz-Sobolev Spaces}\label{subsec:Orlicz-Sobolev Spaces}\noindent
\\
In this subsection, 
we recall some facts on Orlicz-Sobolev functions, which shall be necessary later. Further details can be found in textbooks e.g. \cite{Kuf-O-F, Rao-Ren}. 
\begin{Def}[Young function]\label{def:young}
 If $ \psi :[0,\infty) \to [0,\infty) $ is an non-decreasing, left continuous function with 
$\psi(0) = 0 $ and $ \psi(s)>0 $ for all $ s >0 $, then any function 
$\Psi : [0,\infty) \to [0,\infty] $ of the form 
\begin{equation}\label{eq:young}
  \Psi(t) = \int^t_0 \psi(s)\, ds  
\end{equation}
is called a \textit{Young function}. 
A continuous Young function $\Psi : [0,\infty) \to [0,\infty) $ 
satisfying  
$\Psi(t) = 0$ iff $t = 0,\ \lim_{t\to \infty}\Psi(t)/t  =  \infty$ and $ 
\lim_{t \to 0}\Psi(t)/t=  0 $, is called 
\textit{N-function}. 
\end{Def}
There are several different definitions available in various references.  
However, within a slightly restricted range of functions (as in our case), 
all of them are equivalent. We refer to the book of Rao-Ren \cite{Rao-Ren}, 
for a more general discussion.
\begin{Def}[Conjugate]\label{def:conj}
The \textit{generalised inverse} of a montone function $\psi$ is defined as 
$ \psi^{-1}(t)  :=  \inf\{s \geq 0\ |\ \psi(s) >t \} $.  
Given any Young function $\Psi(t)  =  \int^t_0 \psi(s) ds $, 
its 
\textit{conjugate} function $ \Psi^* : [0,\infty) \to [0,\infty]$ is defined as
\begin{equation}\label{eq:young comp}
 \Psi^* (s)  :=  \int^s_0 \psi^{-1}(t)\, dt 
\end{equation}
and $(\Psi,\Psi^*) $ is called a \textit{complementary pair}, which 
is \textit{normalised} if 
$ \Psi(1) +\Psi^*(1) = 1$. 
\end{Def}
A Young function $\Psi$ is convex, increasing, left continuous and  
satisfies $ \Psi(0) = 0 $ and $ \lim_{t\to\infty}\Psi(t) = \infty $. 
The generalised inverse of $\Psi$ is right continuous, increasing and coincides with the usual inverse 
when $\Psi$ is continuous and strictly increasing.
In general, the  
inequality 
\begin{equation}\label{eq:inv ineq}
\Psi(\Psi^{-1}(t))\leq t \leq \Psi^{-1}(\Psi(t))
\end{equation}
is satisfied for all $ t \geq 0 $ and equality holds when 
$\Psi(t)$ and $\Psi^{-1}(t) \in (0,\infty)$. 
 It is also evident that that the conjugate function $\Psi^*$ is also a Young 
 function, $\Psi^{**} = \Psi$
and for any constant $c>0$, we have $ (c\,\Psi)^*(t) = c\,\Psi^*(t/c)$. Here are two standard examples of 
complementary pair of Young functions.  
\begin{enumerate}
 \item $ \Psi(t) = t^p/p$ and $ \Psi^*(t) = t^{p^*}/p^*$ when 
 $ 1< p, p^*< \infty$ and $ 1/p +1/p^* = 1 $.
 \item $\Psi(t) = (1+t)\log(1+t)-t$ and $ \Psi^*(t) = e^t-t-1 $.
 \end{enumerate}
\begin{Lem}\label{lem:comp prop}
 If $(\Psi,\Psi^*) $ is a complementary pair of N-functions, then for any $t >0$ we have 
 \begin{equation}\label{eq:comp prop} 
  \Psi^*\left(\frac{\Psi(t)}{t}\right)\leq \Psi(t).
 \end{equation}
\end{Lem}
\begin{proof}
 Let $\Psi(t) =  \int^t_0 \psi(s) ds $. From mean value theorem, there exists $ s_0\in (0,t]$ such that 
 $$  \psi(s_0)= \frac{1}{t}\int^t_0 \psi(s)\,ds = \frac{\Psi(t)}{t}  $$
 for every $t>0$. 
 Using definition \eqref{eq:young comp} and mean value theorem again, we find 
 that 
 there exist $r_0 \in(0,\psi(s_0))$, such that we have 
 \begin{equation*}
  \Psi^*\left(\frac{\Psi(t)}{t}\right)= \ \int^{\Psi(t)/t}_0 \inv{\psi}(r)\,dr
  = \ \frac{\Psi(t)}{t}\,\inv{\psi}(r_0) .
 \end{equation*}
Since $\psi$ and $\inv{\psi}$ are non-decreasing functions, hence 
$\inv{\psi}(r_0) \leq \inv{\psi}(\psi(s_0)) = s_0 \leq t$. Using this on the above,  one easily gets \eqref{eq:comp prop}, to complete the proof. 
\end{proof}

The following Young's inequality is well known. We refer to \cite{Rao-Ren} for a proof. 
\begin{Thm}[Young's Inequality]\label{thm:YI}
 Given a Young function $\Psi(t) =  \int^t_0 \psi(s) ds $, we have the following for all $s,t >0$; 
 \begin{equation}\label{eq:YI}
 st \,\leq\,   \Psi(s)  + \Psi^*(t)
\end{equation}
and equality holds iff $ t = \psi(s) $ or $ s = \inv{\psi}(t)$.
\end{Thm}
\begin{Def}[Doubling function]\label{def:doubling}
The Young function $\Psi$ is called \textit{doubling} if there exists 
a constant $C_2>0 $ such that for all $t \geq 0$, we have 
$$\Psi(2t)\leq C_2\, \Psi(t). $$
\end{Def}
In the growth and ellipticity condition \eqref{eq:str cond}, the 
structure function $g$ satisfying \eqref{eq:g prop}, is a doubling function. 
Its doubling constant $C_2 = 2^{g_0}$ (see Lemma \ref{lem:gandG prop} below). 
Henceforth, we restrict to Orlicz spaces of doubling functions, thereby 
avoiding unnecessary technicalities. 
\begin{Def}\label{def:Orl space}
Let $ \Omega\subset\R^m $ be open and $\mu$ be a $\sigma$-finite measure on $\Om$. For a doubling Young function $\Psi$, the 
\textit{Orlicz space} $L^\Psi(\Omega,\mu)$ is defined as the vector space generated by the set
$ \{u: \Omega \to \R\ |\ u\ \text{measurable},\ \int_\Omega 
\Psi(|u|)\,d\mu < \infty\} $. The space is equipped with the following 
\textit{Luxemburg norm} 
\begin{equation}\label{eq:lux norm}
 \|u\|_{L^\Psi(\Omega,\mu)}  :=  \inf\Big\{ k >0 : \int_\Omega 
\Psi\left(\frac{|u|}{k}\right) \,d\mu \leq 1\Big\} 
\end{equation}
If $\mu $ is the Lebesgue measure, the space 
is denoted by $L^\Psi(\Omega)$ and any $u \in L^\Psi(\Omega)$ is called 
a $\Psi$-integrable function.
\end{Def}
The function 
$u \mapsto \|u\|_{L^\Psi(\Omega,\mu)}$ is lower semi continuous and $L^\Psi(\Omega,\mu)$ is a
Banach space with the norm in \eqref{eq:lux norm}. 
%
The following theorem is a generalised version of H\"older's inequality, which 
follows easily from   
the Young's inequality \eqref{eq:YI}, see \cite{Rao-Ren} or \cite{Tuo}. 
\begin{Thm}[H\"older's Inequality]\label{thm:gen holder}
 For every $ u \in L^\Psi(\Omega,\mu)$ and $ v \in L^{\Psi^*}(\Omega,\mu)$, we have 
 \begin{equation}\label{eq:gen holder}
 \int_\Om |uv|\,d\mu\leq 2\, \|u\|_{L^\Psi(\Omega,\mu)}\|v\|_{L^{\Psi^*}(\Omega,\mu)}
\end{equation}
\end{Thm}
\begin{Rem}
The factor $2$ on the right hand side of the above, can be dropped if $(\Psi, \Psi^*)$ is normalised and one is replaced by 
$\Psi(1)$ in the definition \eqref{eq:lux norm} of Luxemburg norm. 
\end{Rem}
The \textit{Orlicz-Sobolev 
space} $W^{1,\Psi}(\Om)$ 
can be defined similarly by $ L^\Psi$ norms of the function and 
its gradient, see \cite{Rao-Ren}, that resembles $W^{1,p}(\Om)$ for the 
special case of $\Psi(t) =t^p$. 
But here for $\Om\subset \hh^n$, we require the notion of  
\textit{Horizontal Orlicz-Sobolev spaces}, analoguous to the horizontal Sobolev spaces defined in the previous subsection.
\begin{Def}\label{def:HOS space}
 We define the space 
 $HW^{1,\Psi}(\Omega) = \{u\in L^\Psi(\Omega)\ |\ \Xu\in L^\Psi(\Omega,\R^{2n})\}$ for an open set $\Om\subset \hh^n$ and a  
 doubling Young function $\Psi$, along with the norm 
 $$ \|u\|_{HW^{1,\Psi}(\Omega)} :=  
\|u\|_{L^\Psi(\Omega)}+ \|\Xu\|_{L^\Psi(\Omega, \R^{2n})} ;$$
the spaces $HW^{1,\Psi}_\loc(\Om),\ HW^{1,\Psi}_0(\Om)$ are  defined, 
similarly as earlier.
\end{Def}
We remark that, all these notions can be defined for a general metric space, equipped with a doubling measure and upper gradient. More details of these can be found in \cite{Tuo}.

\subsection{Sub-elliptic equations}\noindent
\\
Here, we discuss the known results on existence and uniqueness of weak 
solutions of the equation \eqref{eq:eq}. Using the notation of horizontal divergence, 
we rewrite \eqref{eq:eq} as
\begin{equation}\label{eq:maineq}
-\dvh (\A(\X u)) =  0 \ \ \text{in}\ \Om,
\end{equation}
where $\A:\R^{2n}\to \R^{2n}$ satisfies \eqref{eq:str cond} and the matrix  
$D\A(z)$ is symmetric. Now, we enlist monotonicity and doubling properties of the structure 
function $g$, in the following lemma.
\begin{Lem}\label{lem:gandG prop} Let $ g\in C^1([0,\infty)) $ be a function that satisfies \eqref{eq:g prop} for some constant $g_0>0$ and 
$g(0) = 0$. If $ G(t) = \int^t_0 g(s)ds $, then the following holds. 
\begin{align}
\label{eq:gG1}&(1) \  \ G \in C^2([0,\infty))\ \text{is convex}\,;\\
\label{eq:gG2}&(2)\ \ tg(t)/(1+g_0)\leq G(t)\leq tg(t)\ \ \forall\ \ t\geq 0;\\
\label{eq:gG3}&(3)\  \ g(s)\leq g(t)\leq (t/s)^{g_0}g(s) \ \ \forall\ \ 0\leq s<t;\\
\label{eq:gG4}&(4) \ \ G(t)/t \ \text{is an increasing function}\ \ \forall\ \ t>0;\\
\label{eq:gG5}&(5) \ \ tg(s)\leq tg(t) + sg(s)\ \ \forall\ \ t,s \geq 0.
\end{align}
\end{Lem}
The proof of the above lemma is trivial  
(see Lemma 1.1 of \cite{Lieb--gen}), so we omit it.
Notice that \eqref{eq:gG3} implies that $g$ is increasing and 
doubling, with $g(2t) \leq 2^{g_0} g(t)$. In fact, it is easy to see that, 
\eqref{eq:g prop} implies $t\mapsto g(t)/t^{g_0}$ is decreasing and 
$t\mapsto g(t)/t^\delta$ is increasing. Thus, 
\begin{equation}\label{eq:gdoub}
 \min\{\alpha^\delta,\alpha^{g_0}\}g(t)\leq g(\alpha t)\leq 
 \max\{\alpha^\delta,\alpha^{g_0}\}g(t)\quad\text{for all}\ \alpha, t\geq 0. 
\end{equation}

Here onwards, we fix the following notations, 
\begin{equation}\label{eq:GandF}
\F(t)  :=  g(t)/t \quad \text{and}\quad G(t)  :=  \int^t_0 g(s)\,ds. 
\end{equation} 
Thus, $\F$ and $G$ are also doubling functions and $G$ is a Young function.
Now we restate the structure condition \eqref{eq:str cond}. For every 
$z,\xi\in\R^{2n}$, we have that
\begin{equation}\label{eq:str cond diff}
 \begin{aligned}
 \F(|z|)|\xi|^2 \leq \,&\inp{D\A(z)\,\xi}{\xi}\leq L\,\F(|z|)|\xi|^2;\\
 &|\A(z)|\leq L\,|z|\F(|z|). 
\end{aligned} 
\end{equation}
\begin{Def}\label{def:weak soln}
 Any $ u\in HW^{1,G}(\Om)$ is called a weak solution of the equation \eqref{eq:maineq} if 
 for every $ \varphi \in C^\infty_0(\Om)$, we have that 
 \begin{equation}\label{eq:weak soln}
  \int_\Om \inp{\A(\X u)}{\X \varphi}\dx= 0. 
 \end{equation}
 In addition, for all non-negative $ \varphi \in C^\infty_0(\Om)$, if the integral above is positive (resp. negative) then $u$ is called a weak supersolution (resp. subsolution) of the equation \eqref{eq:maineq}. 
 \end{Def}
Monotonicity of the operator $\A$ is required for existence of weak solutions. 
This follows from the structure condition  
\eqref{eq:str cond diff}. First, notice that, 
from \eqref{eq:str cond diff}
\begin{align*}
\inp{\A(z)-\A(w)}{z-w}&= \int^1_0\inp{D\A\big(w+t(z-w)\big)(z-w)}{(z-w)} \,dt\\
&\geq |z-w|^2\int^1_0 \F(|w+t(z-w)|)\,dt,
\end{align*}
for any $ z, w \in \R^{2n}$. Now, it is possible to show that
\begin{align*}
 |z|/2\leq \, &\abs{tz+(1-t)w}\leq 3|z|/2
 \ \ \quad\quad\quad\ \text{if}\ |z-w|\leq 2|z|,\ t\geq 3/4 ,\\ 
 |z-w|/4\leq \, &\abs{tz+(1-t)w}\leq 3|z-w|/2
 \ \ \ \quad\text{if}\ |z-w|> 2|z|,\ t\leq 1/4,
\end{align*}
with appropriate use of triangle inequality. Combining the above inequalities and using the doubling property, we 
have the following monotonicity inequality 
\begin{equation}\label{eq:monotone}
 \begin{aligned}
  \inp{\A(z)-\A(w)}{z-w}\geq c(g_0)\ 
\begin{cases}
  |z-w|^2 \,\F(|z|)\ &\text{if}\ |z-w|\leq 2|z| \\
 |z-w|^2 \,\F(|z-w|) \ \ &\text{if}\ |z-w|> 2|z|
\end{cases}
 \end{aligned}
\end{equation} 
and therefore the following ellipticity condition
\begin{equation}
\label{eq:elliptic} \inp{\A(z)}{z}\geq c(g_0)\,|z|^2\F(|z|) \geq c(g_0) G(|z|).
\end{equation} 
\begin{Rem}\label{rem:comparison}
The inequality in \eqref{eq:monotone} is reminiscent of the monotonicity inequality for the $p$-laplacian operator. Precisely, when $\A(z) = |z|^{p-2}z $ for $ 1<p<\infty$, we have
 \begin{align}
\label{eq:special case} \ \  \left(|z|^{p-2}z  - |w|^{p-2}w\right)\cdot(z-w) \geq c(p) 
\begin{cases}
 |z-w|^2(|z|+|w|)^{p-2}\ \ &\text{if}\ \ 1<p<2\\
 |z-w|^p \ \ &\text{if}\ \ p\geq 2
\end{cases}
\end{align}
and from this, one can also derive \eqref{eq:monotone} for this special case. 
\end{Rem} 
\begin{Thm}[Existence]\label{thm:existence}
If $ u_0 \in HW^{1,G}(\Om)$ is a given function and the operator $\A$ has the 
structure condition \eqref{eq:str cond diff}, then 
there exists a unique weak solution $u\in HW^{1,G}(\Om)$ for 
the Dirichlet problem
\begin{equation}\label{eq:dirichlet prob}
 \begin{cases}
  -\dvh (\A(\X u))= \ 0\ \ \text{in}\ \Om;\\
 \ \ \ u - u_0\in HW^{1,G}_0(\Om).
 \end{cases}
\end{equation}
\end{Thm}
The proof of this theorem is a standard variant of that for the Euclidean setting and relies on literature of variational inequalities for monotone operators by Kinderlehrer and Stampacchia \cite{Kind-Stamp}. 
Similarly as the proof of Theorem 17.1 in \cite{Hein-Kilp-Mar}, 
it is possible to show
that 
there exists $u \in\mathcal{K} $ satisfying the 
variational inequality
$$ \int_\Om \inp{\A(\X u)}{\X w-\X u}\dx \geq 0 $$ 
for all $w \in \mathcal{K}$, where 
$ \mathcal{K} = \{ v\in  HW^{1,G}(\Om) \,|\, v - u_0\in HW^{1,G}_0(\Om)\} $. 
Arguing with $w = u\pm\varphi$ for 
any $\varphi \in C^\infty_0(\Om)$, it is easy to see that $u$ satisfies 
\eqref{eq:weak soln} and hence, is a weak solution of \eqref{eq:dirichlet prob}.
The conditions for existence of $u$, can be established from the 
monotonicity \eqref{eq:monotone}.

The uniqueness, follows from the following comparison principle, which can be 
easily proved by choosing an appropriate test function on \eqref{eq:maineq} and 
using monotonicity.
\begin{Lem}[Comparison Principle]\label{lem:comparison principle}
Given $ u,v\in HW^{1,G}(\Om) $, if $u$ and $v$ respectively are weak super and subsolution 
of the equation \eqref{eq:maineq} and 
$u \geq v$ on $\del\Om$ in the trace sense, then we have 
$ u \geq v$ a.e. in $\Om$.
\end{Lem}

We would also require that, the weak solution of the 
Dirichlet problem \eqref{eq:dirichlet prob} is Lipschitz with respect to CC-metric, if it has smooth boundary value in strictly convex domain. The proof of this resembles the Hilbert-Haar theory in the Euclidean setting. Actually, this is the only place where we require that $D\A$ is symmetric. 

Consider a bounded domain $D\subset \R^{2n+1}$ which is convex and 
there exists a constant $\eps_0>0$ such that the following holds:  
for every $y\in \del D$, 
there exists $b(y)\in \R^{2n+1}$ with $|b(y)| =1$, such that  
\begin{equation}\label{eq:str conv}
b(y)\cdot (x-y) \geq \eps_0 |x-y|^2
\end{equation}
for all $x\in \bar D$. 
Here $(\cdot)$ is the Euclidean inner product and 
$|.|$ is the Euclidean norm of $\R^{2n+1}$. The following 
theorem shows existence of Lipschitz continuous solutions of \eqref{eq:maineq}. 
The statement and the proof of this theorem, are the same as those of Theorem 5.1 of \cite{Zhong}. For sake of completeness, we provide a sketch of the proof.
\begin{Thm}\label{prop:lip existence}
  Let $D \subset \hh^n$ be a bounded and convex domain satisfying 
  \eqref{eq:str conv} for some $\eps_0>0$. 
  Given $u_0\in C^2(\bar D)$, if $u\in HW^{1,G}(D)$ is the weak solution of
the Dirichlet problem
\begin{equation}\label{eq:Dprob}
\begin{cases}
&\dvh (\A(\Xu))=0 \quad  \text{ in } D;\\
& u-u_0\in HW^{1,G}_0(D).
\end{cases}
\end{equation}
then 
there exists a constant  
 $ M = M\big(n,\,\eps_0,\, 
 \|\gr u_0\|_{L^\infty(\bar D)}+\|D^2 u_0\|_{L^\infty(\bar D)},\,\diam(D) \big)>0$, such that we have
 $$ \norm{\X u}_{L^\infty(D)} \leq M .$$
 \end{Thm}
\begin{proof}
This proof is the same as that of Theorem 5.1 in \cite{Zhong}, with minor changes. Here, we provide a brief outline for the reader's convenience. It is enough to 
show that 
\begin{equation}\label{eq:toshow}
|u(x)-u(y)|\leq M d(x,y)\quad \forall\ x,y\in\bar D
\end{equation}
for some constant $ M = M\big(n,\,\eps_0,\, 
 \|\gr u_0\|_{L^\infty(\bar D)}+\|D^2 u_0\|_{L^\infty(\bar D)},\,\diam(D) \big)>0$. 
 
To this end, we fix $y\in \del D$ and consider the barrier functions 
$$ L^\pm(x)= u_0(y)+ [\gr u_0(y)\pm K\,b(y)]\cdot (x-y), 
\ \ \text{where}
\ K = \frac{(2n+1)^2}{2\eps_0}\|D^2 u_0\|_{L^\infty(\bar D)}$$ 
Taking $\xi$ as an appropriate point between $x$ and $y$ and 
using the Taylor's 
formula followed by the condition \eqref{eq:str conv}, we obtain
\begin{align*}
u_0(x) &= u_0(y) + \gr u_0(y)\cdot (x-y)
+ \frac{1}{2}D^2u_0(\xi)(x-y)\cdot(x-y) \\
&\leq u_0(y) + \gr u_0(y)\cdot (x-y)
+ K\eps_0 |x-y|^2 \leq L^+(x)
\end{align*}
and similarly $L^-(x)\leq u_0(x)$ for 
all $ x\in \bar D$. Thus, if $u\in HW^{1,G}(D)$ is the weak solution of \eqref{eq:Dprob}, since $u_0$ is continuous on the boundary, 
we have 
\begin{equation}\label{eq:bdineq}
L^-(x) \leq u(x) \leq L^+(x)\quad \forall\ x\in\del D 
\end{equation}
upto a continuous representative of $u$. 
Now, letting $b(y)=(b_1(y),\ldots,b_{2n}(y), b_t(y)) \in \R^{2n}\times \R$ and explicit computations using 
\eqref{eq:Xdef}, we find that $\XX L^\pm$ is skew-symmetric. 
Precisely, 
$$ \XX L^\pm(x) = \ \frac{1}{2}\big[\del_t u_0(y)\pm K\,b_t(y)\big]\begin{pmatrix}
    0 & I_n\\
    -I_n & 0
   \end{pmatrix}  $$
for every $x \in \bar D$. Since the 
matrix $D\A(z)$ has been assumed to be symmetric, we have 
$\dvh [\A(\X L^\pm)] = \Tr\big(D\A(\X L^\pm)^T\XX L^\pm\big) = 0 $. 
Thus, $L^\pm$ are solutions of the 
equation \eqref{eq:Dprob}. 
This, together with \eqref{eq:bdineq} and comparison principle (Lemma \ref{lem:comparison principle}), implies   
$$ L^-(x)\leq u(x)\leq L^+(x) \quad \forall\ x\in D.$$ 
Since $L^\pm$ are Lipschitz and $L^\pm(y)=u(y)$, it is evident that  
there exists $M>0$ such that 
\begin{equation}\label{eq:lip}
  -Md(x,y) \leq u(x)-u(y) \leq M d(x,y) \quad \forall\ x\in\bar D,y\in\del D
\end{equation}
Now, we need the fact that if $u$ is a Lipschitz solution of \eqref{eq:Dprob}, then the following holds, 
\begin{equation}\label{eq:haar}
\sup_{x,y\in \bar D}\left(\,\frac{|u(x)-u(y)|}{d(x,y)}\, \right)= 
\sup_{x\in\bar D,\, y\in \del D}\left(\,\frac{|u(x)-u(y)|}{d(x,y)}\,\right).
\end{equation}
We refer to \cite{Zhong} for a proof of 
\eqref{eq:haar}. From \eqref{eq:lip} and \eqref{eq:haar}, we 
immediately get \eqref{eq:toshow} and the proof is finished. 
\end{proof}
 
\section{Local Boundedness of Horizontal gradient}\label{sec:LocalB}

We shall prove Theorem \ref{thm:mainthm} in this section. 
In the following three subsections we prove some Caccioppoli type inequalities of the horizontal and vertical vector fields, under two supplementary assumptions 
(see \eqref{eq:ass1} and \eqref{eq:ass2} below). The proof of Theorem \ref{thm:mainthm} is given at the end of this section, where we remove both assumptions one by one.

Throughout this section, we denote $ u\in HW^{1,G}(\Om) $ as a weak solution of \eqref{eq:maineq}. We assume the growth and ellipticity conditions  
\eqref{eq:str cond diff}, retaining the notations 
\eqref{eq:GandF}.

Now we make two supplementary assumptions:
\begin{align}\label{eq:ass1}
&(1)\  \text{there exists}\ m_1,m_2>0 \ \text{such that}
\ \lim_{t \to 0}\F(t)=m_1\ \text{and}\ \lim_{t \to \infty}\F(t)=m_2;\\ 
\label{eq:ass2}
&(2)\  \text{there exists}\ M>0 \ \text{such that}\ \|\X u\|_{L^\infty(\Om)} \leq M.
\end{align} 
The purpose of the assumptions, is to ensure weak-differentiability of weak solutions 
of the equation \eqref{eq:maineq}. Since $\F(t)=g(t)/t$ and $g$ is monotonic, 
$\F$ has possible singularities at $t\to 0$ or $t\to \infty$ (or both). The 
assmption \eqref{eq:ass1} avoids this and consequently, the structure condition 
\eqref{eq:str cond diff}
together with \eqref{eq:ass1} and \eqref{eq:ass2}, imply
\begin{equation}\label{eq:str cond new}
 \begin{aligned}
 \inv{\nu}|\xi|^2 \leq \,&\inp{D\A(\X u)\,\xi}{\xi}\leq \nu\,|\xi|^2;\\
 &|\A(\X u)|\leq \nu\,|\X u|,
 \end{aligned}
\end{equation}
for some $\nu = \nu(g_0,L, M,m_1,m_2)>0$. 
Thus, the equation \eqref{eq:maineq} with \eqref{eq:str cond new}, 
satisfies the conditions  
considered by Capogna \cite{Cap--reg}. 
From Theorem 1.1 and Theorem 3.1 of \cite{Cap--reg}, we get 
\begin{equation}\label{eq:Cap reg}
\X u \in HW^{1,2}_\loc(\Om,\R^{2n})\cap C^{0,\alpha}_\loc(\Om,\R^{2n}) ,
\ \  Tu  \in HW^{1,2}_\loc(\Om)\cap C^{0,\alpha}_\loc(\Om) .
\end{equation}
However, every apriori estimates that follow in this section, are independent of the constants $M,m_1,m_2$. 
This enables us to remove both the assumptions \eqref{eq:ass1} and \eqref{eq:ass2}, in the end.

\subsection{Caccioppoli type inequalities}\label{subsec:Caccioppoli type inequalities}\noindent
\\
By virtue of \eqref{eq:Cap reg}, we can weakly differentiate the equation \eqref{eq:maineq} and obtain the equations satisfied by $X_lu$ and $Tu$ in the weak sense. This is 
shown in the following two lemmas.
\begin{Lem}\label{lem: Tu}
If $ u \in HW^{1,G}(\Om) $ is a weak solution of \eqref{eq:maineq}, then $Tu$ is a weak solution of 
\begin{equation}\label{eq:Tu}
\sum_{i,j=1}^{2n}X_i(D_j\A_i(\X u)X_j(Tu))=0.
\end{equation}
\end{Lem}
The proof of the above lemma is quite easy and similar to Lemma 3.2 in \cite{Zhong}. So, we omit the proof. The following lemma is similar to 
Lemma 3.1 in \cite{Zhong}.
\begin{Lem}\label{lem:Xu}
If $ u \in HW^{1,G}(\Om) $ is a weak solution of \eqref{eq:maineq}, then
for any $ l \in \{1,\ldots, n\}$, we have that $ X_lu$ is weak solution of 
\begin{equation}\label{eq:Xu}
\sum_{i,j=1}^{2n} X_i(D_j\A_i(\X u)X_jX_lu)
+\sum_{i=1}^{2n}X_i(D_i\A_{n+l}(\X u)Tu) +T(\A_{n+l}(\X u))= 0
\end{equation}
and similarly, $X_{n+l}u $ is weak solution of
\begin{equation}\label{eq:Xu other}
\sum_{i,j=1}^{2n} X_i(D_j\A_i(\X u)X_jX_{n+l}u)
-\sum_{i=1}^{2n}X_i(D_i\A_{l}(\X u)Tu) -T(\A_{l}(\X u))= 0.
\end{equation}
\end{Lem}\noindent

\begin{proof}
 We only prove \eqref{eq:Xu}, the proof of
\eqref{eq:Xu other} is similar. 
Let $l\in \{1,2,\ldots,n\}$ and 
$\varphi\in
C^\infty_0(\Omega)$ be fixed. We choose test function $X_l\varphi$ in \eqref{eq:maineq} to get
$$ \int_\Om \sum_{i=1}^{2n} \A_i(\X u)X_iX_l\varphi\dx=0. $$
Recalling the commutation relation \eqref{eq:comm} and using integral by parts, we obtain 
\begin{equation}\label{eq:weak}
\begin{aligned}
 0
=& \int_\Om \sum_{i=1}^{2n} \A_i(\X u)X_lX_i\varphi\dx -\int_\Om \A_{n+l}(\X u) T\varphi\dx\\
=& -\int_\Om \sum_{i=1}^{2n} X_l (\A_i(\X u)X_i\varphi)dx+ \int_\Om T(\A_{n+l}(\X u))\varphi\dx.\\
\end{aligned}
\end{equation}
From \eqref{eq:comm} again, notice that for every $i\in \{1,2,\ldots,2n\}$, 
\begin{equation}\label{eq:conn}
X_l(\A_i (\X u))=  \sum_{j=1}^{2n} D_j\A_i(\X u)X_jX_lu+ D_i\A_{n+l}(\X u)Tu. 
\end{equation}
Thus, \eqref{eq:weak} and \eqref{eq:conn} together completes the proof. 
\end{proof}
The following Caccioppoli type inequality for $Tu$ is quite standard. We provide a proof for the reader's convenience. 
\begin{Lem}\label{lem:cacci T}
For any $\gamma\geq 0$ and $\eta\in C^\infty_0(\Omega)$, there exists 
$c = c(n,g_0,L)>0$ such that 
\begin{equation*}
\int_\Omega \eta^2G(| Tu|)^{\gamma+1}\F(|\X u|) |\X( Tu)|^2\,
dx\leq \frac{c}{(\gamma+1)^2}\int_\Omega G(| Tu|)^{\gamma+1}\F(|\X u|)|  Tu|^2|\X\eta
|^2 dx.
\end{equation*}
\end{Lem}
\begin{proof}
For some fixed $ \eta \in C^\infty_0(\Omega)$ and $ \gamma\geq 0 $, we choose test funcion $$\varphi=\eta^2G(| Tu|)^{\gamma+1} Tu $$ in the equation \eqref{eq:Tu} to 
get 
\begin{equation*}
  \begin{aligned}
\sum_{i,j=1}^{2n}\int_\Omega \eta^2 &G(| Tu|)^{\gamma+1}D_j\A_i(\X u)X_j(Tu)X_i(Tu)\, dx \\
   +\ (\gamma+1)&\sum_{i,j=1}^{2n}\int_\Omega  \eta^2 G(| Tu|)^{\gamma}g(| Tu|)| Tu|D_j\A_i(\X u)X_j(Tu)X_i(Tu)\dx \\
   &=  -2\sum_{i,j=1}^{2n}\int_\Omega \eta\, G(| Tu|)^{\gamma+1}Tu D_j\A_i(\X u)X_j(Tu)X_i\eta\, dx  .
  \end{aligned}
 \end{equation*}
We use the condition \eqref{eq:gG2} on the first term and 
then use the structure condition \eqref{eq:str cond diff}, to estimate both 
sides of the above equality. 
We obtain
\begin{align*}
 \int_\Omega \eta^2 G(| Tu&|)^{\gamma+1}\weight |\X(Tu)|^2\dx\\
 &\leq \frac{c}{(\gamma+1)}\int_\Omega |\eta| G(| Tu|)^{\gamma+1}|Tu|
 \weight |\X(Tu)||\X\eta|\dx\\
 &\leq\ c\tau\int_\Omega \eta^2 G(| Tu|)^{\gamma+1}\weight |\X(Tu)|^2\dx\\
 &\qquad\quad+ \frac{c}{\tau(\gamma+1)^2}\int_\Omega G(| Tu|)^{\gamma+1}
 \F(|\X u|)|Tu|^2|\X\eta|^2 \dx,
 \end{align*}
where we have used Young's inequality to obtain the latter inequality of the above. With the choice of a small enough $\tau>0$, the proof is finished.
\end{proof}
The following Caccioppoli type inequality for the horizontal vector fields is more involved than the above, due to the non-commutativity \eqref{eq:comm}. 
For the case of $p$-Laplace equations, similar inequalities
have been proved before, using difference quotients for $ 2\leq p\leq 4$ in \cite{Min-Z-Zhong} and directly, for $1<p<\infty$ in \cite{Zhong}.
\begin{Lem}\label{lem:cacci Xu}
For any $\gamma\geq 0$ and $\eta\in C^\infty_0(\Omega)$, there exists 
$c = c(n,g_0,L)>0$ such that 
\begin{equation*}
\begin{aligned}
\int_{\Om}\eta^2G(|\X u|)^{\gamma+1}\weight|\XX u|^2 \dx\leq &
\ c\int_\Om G(|\X u|)^{\gamma+1}|\X u|^2\weight\left(| \X \eta|^2+|\eta T\eta|\right) dx\\
&+ c\,(\gamma+1)^4\int_\Om\eta^2G(|\X u|)^{\gamma+1}\weight| Tu|^2\dx.
\end{aligned}
\end{equation*}
\end{Lem}

\begin{proof}
We fix $ l\in\{1,\ldots,n\}$ and $\eta\in C^\infty_0(\Omega)$. Now, we choose 
$\varphi_l =  \eta^2G(|\X u|)^{\gamma+1}X_lu $ as a test function  
in \eqref{eq:Xu} and obtain the following,
\begin{equation}\label{eq:xu1}
\begin{aligned}
\sum_{i,j=1}^{2n}&\int_{\Omega}\eta^{2} G(|\X u|)^{\gamma+1} 
D_j\A_i(\X u)X_jX_luX_iX_lu\dx\\
& +\ (\gamma+1)\sum_{i,j=1}^{2n}\int_\Omega \eta^{2} G(|\X u|)^{\gamma} 
 X_lu D_j\A_i(\X u)X_jX_lu X_i(G(|\X u|))\dx\\
&\ \ =  -2\sum_{i,j=1}^{2n}\int_{\Omega}\eta\,G(|\X u|)^{\gamma+1} X_luD_j\A_i(\X u)X_jX_luX_i\eta\dx\\
&\qquad -\sum_{i=1}^{2n}\int_{\Omega}D_i\A_{n+l}(\X u)X_i\varphi_l Tu\ dx\\
&\qquad  +\int_{\Omega} T(\A_{n+l}(\X u))\,\varphi_l\dx\\
&\ \ = J_{1,l}+ J_{2,l}+ J_{3,l}.
\end{aligned}
\end{equation}
Similarly, we choose $\varphi_{n+l} =  \eta^2G(|\X u|)^{\gamma+1}X_{n+l}u $ 
in \eqref{eq:Xu other} to get
\begin{equation}\label{eq:xu2}
\begin{aligned}
\sum_{i,j=1}^{2n}&\int_{\Omega}\eta^{2} G(|\X u|)^{\gamma+1} 
D_j\A_i(\X u)X_jX_{n+l}uX_iX_{n+l}u\dx\\
& +\ (\gamma+1)\sum_{i,j=1}^{2n}\int_\Omega \eta^{2} G(|\X u|)^{\gamma} 
 X_{n+l}u D_j\A_i(\X u)X_jX_{n+l}u X_i(G(|\X u|))\dx\\
&\ \ =  -2\sum_{i,j=1}^{2n}\int_{\Omega}\eta\,G(|\X u|)^{\gamma+1} 
X_{n+l}u D_j\A_i(\X u)X_jX_{n+l}uX_i\eta\dx\\
&\qquad +\sum_{i=1}^{2n}\int_{\Omega}D_i\A_{l}(\X u)X_i\varphi_{n+l} Tu\ dx\\
&\qquad  -\int_{\Omega} T(\A_{l}(\X u))\,\varphi_{n+l}\dx\\
&\ \ = J_{1,n+l}+ J_{2,n+l}+ J_{3,n+l}.
\end{aligned}
\end{equation} 
We shall add \eqref{eq:xu1} and \eqref{eq:xu2} and estimate both sides. First, 
notice that
$$ X_i(G(|\X u|)) = \frac{g(|\X u|)}{|\X u|} \sum_{k=1}^{2n}X_ku X_iX_k u .$$ 
We shall use the above along with \eqref{eq:gG2}. Adding \eqref{eq:xu1} and \eqref{eq:xu2} and using the structure condition \eqref{eq:str cond diff}, we  obtain that
\begin{equation}\label{eq:J0}
 \sum_{l=1}^{2n} (J_{1,l}+ J_{2,l}+ J_{3,l})\, \geq \int_\Om \eta^2 
 G(|\X u|)^{\gamma+1}\weight |\XX u|^2\dx
\end{equation}
Now we claim the following, which combined with \eqref{eq:J0} concludes the proof of the 
lemma. $ \textbf{Claim\,:}$ For every $ k \in \{1,2,3\},l \in \{1,\ldots,2n\}$ and some $c = c(n,g_0,L)>0$, we have  
\begin{equation}\label{eq:claim1}
\begin{aligned}
|J_{k,l}| \leq \frac{1}{12n}\int_{\Om}&\eta^2G(|\X u|)^{\gamma+1}\weight|\XX u|^2 \dx\ \\
&+\, c\int_\Om G(|\X u|)^{\gamma+1}|\X u|^2\weight\left(| \X \eta|^2+|\eta T\eta|\right) \dx\\
&+\, c\,(\gamma+1)^4\int_\Om\eta^2G(|\X u|)^{\gamma+1}\weight| Tu|^2\dx.
\end{aligned} 
\end{equation}
We prove the claim by estimating each $J_{k,l}$ in \eqref{eq:xu1} and \eqref{eq:xu2}, using \eqref{eq:str cond diff}. 

For the first term, we obtain 
\begin{equation*}
 |J_{1,l}| \leq c\int_\Om |\eta|G(|\X u|)^{\gamma+1}|\X u|\F(|\X u|)|\XX u||\X \eta|\dx
\end{equation*}
and the claim \eqref{eq:claim1} for $J_{1,l}$, follows from Young's inequality. 

We calculate $\X\varphi_l$ and 
similary estimate the second term using \eqref{eq:str cond diff}, to get 
\begin{equation}\label{eq:j2}
\begin{aligned}
|J_{2,l}|\leq c\int_{\Omega}\eta^2 &G(|\X u|)^{\gamma+1}\weight 
|Tu||\XX u| \dx\\
&+\, c\,(\gamma+1)\int_\Om \eta^2 G(|\X u|)^{\gamma}g(|\X u|)|\X u|\F(|\X u|) |Tu||\XX u|\dx \\
&+\, c\int_\Om |\eta| G(|\X u|)^{\gamma+1}|\X u|\F(|\X u|)|\X\eta||Tu|\dx.
\end{aligned}
\end{equation}
Recalling $tg(t)\leq (1+g_0)G(t)$ from \eqref{eq:gG2}, note that the second term of the right hand side of \eqref{eq:j2} can be replaced by 
the first term. Then the claim \eqref{eq:claim1} for $J_{2,l}$, follows by applying Young's inequality on each terms of the above. 

For the third term, we show the estimate only for \eqref{eq:xu1} i.e. for 
$l\in\{1,\ldots,n\}$, since the estimate for the other case is the same.
We first use integral by parts, then we calculate $T\varphi_l$ and obtain 
the following;
\begin{equation*}
\begin{aligned}
J_{3,l} 
=  -\int_{\Omega} &\eta^2G(|\X u|)^{\gamma+1}\A_{n+l}(\X u)X_l(Tu)\dx\\ 
&- (\gamma+1)\int_\Om \eta^2 G(|\X u|)^{\gamma}X_lu\,\A_{n+l}(\X u)T(G(|\X u|))\dx \\
&-\ 2\int_\Om \eta\, G(|\X u|)^{\gamma+1}X_lu\,\A_{n+l}(\X u)T\eta\dx.
\end{aligned}
\end{equation*}
Now, notice that $$ T(G(|\X u|)) = \frac{g(|\X u|)}{|\X u|} \sum_{k=1}^{2n}X_ku 
X_k(Tu)=\weight \sum_{k=1}^{2n}X_ku 
X_k(Tu) .$$
Using this, we carry out integral by parts 
again, for the 
first two terms of $J_{3,l}$ and obtain 
\begin{equation*}
\begin{aligned}
J_{3,l}=\int_{\Omega} &X_l\Big(\eta^2G(|\X u|)^{\gamma+1}\A_{n+l}(\X u)\Big)Tu\dx\\ 
&- (\gamma+1)\int_\Om \sum_{k=1}^{2n}X_k\Big(\eta^2 G(|\X u|)^{\gamma}
\F(|\X u|)X_lu\,\A_{n+l}(\X u)X_k u\Big) Tu\dx \\
&- 2\int_\Om \eta\, G(|\X u|)^{\gamma+1}X_lu\,\A_{n+l}(\X u)T\eta\dx.
\end{aligned}
\end{equation*}
From standard calculations and structure condition 
\eqref{eq:str cond diff}, we get 
\begin{equation}\label{eq:j3}
\begin{aligned}
|J_{3,l}|\leq c\,(&\gamma+1)^2\int_{\Omega}\eta^2 G(|\X u|)^{\gamma+1}\weight 
|Tu||\XX u| \dx\\
&+\,c\,(\gamma+1)^2\int_\Om \eta^2 G(|\X u|)^{\gamma}g(|\X u|)|\X u|\F(|\X u|) |Tu||\XX u|\dx\\
&+\, c\,(\gamma+1)\int_\Om |\eta| G(|\X u|)^{\gamma+1}|\X u|\F(|\X u|)|\X\eta||Tu|\dx\\
&+\, c\int_\Om |\eta|G(|\X u|)^{\gamma+1}|\X u|^2\weight |T\eta| \dx.
\end{aligned}
\end{equation}
Similarly as the estimate of $J_{2,l}$ in \eqref{eq:j2}, we use \eqref{eq:gG2} to combine the first two terms of the right hand side of \eqref{eq:j3}. Then, by applying Young's inequality on all terms except the last one, the claim \eqref{eq:claim1} for 
$J_{3,l}$ follows. Thus, the proof is finished. 
\end{proof}
\subsection{A Reverse type inequality}\label{subsec:Reverse type inequality}\noindent
\\
We follow the technique of Zhong \cite{Zhong} and obtain a reverse type 
inequality for $Tu$ in the following lemma. 
This shall be crucial for obtaining estimates for horizontal and vertical derivatives, later. The following lemma is reminiscent to Lemma 3.5 in \cite{Zhong}.
\begin{Lem}\label{lem:Rev}
For any $\gamma\geq 1$ and all non-negative $\eta\in C^\infty_0(\Omega)$, we have
\begin{equation}\label{eq:Rev}
\begin{aligned}
\int_\Omega \eta^{2} G(\eta &|Tu|)^{\gamma+1} \weight|\XX u|^2\dx\\
\leq c\,&(\gamma+1)^2 
\|\X\eta\|_{L^\infty}^2\int_\Omega G(\eta |Tu|)^{\gamma+1} |Tu|^{-2}
|\X u|^2\weight  | \XX u|^2\dx
\end{aligned}
\end{equation}
for some $c=c(n,g_0,L)>0$. 
\end{Lem}

\begin{proof}
 First, notice that from \eqref{eq:gG2}, we have   
 $G(\eta |Tu|)^{\gamma+1} |Tu|^{-2} \leq \eta^2G(\eta|Tu|)^{\gamma-1}
 g(\eta|Tu|)^2$ for every $\gamma\geq 1$. In other words, the integral in right hand side of \eqref{eq:Rev}, is not singular. 
 
To prove the lemma, we fix $l\in\{1,\ldots,n\}$ and invoke \eqref{eq:weak}, 
i.e. for any $\varphi \in C^\infty_0(\Omega)$ 
$$ \int_\Om \sum_{i=1}^{2n} X_l (\A_i(\X u)X_i\varphi)dx= \int_\Om T(\A_{n+l}(\X u))\varphi\dx.$$
 We choose the test function $\varphi = 
 \eta^2G(\eta |Tu|)^{\gamma+1} X_lu $ in the above. Notice that 
 \begin{align*}
 X_i\varphi =  \eta^2G(\eta |Tu|)^{\gamma+1} &X_iX_lu
 + (\gamma+1)\eta^3 G(\eta |Tu|)^{\gamma}g(\eta |Tu|)X_lu\,X_i(|Tu|)\\
 + &\Big(2\eta G(\eta |Tu|)^{\gamma+1}
 +(\gamma+1)\eta^2 G(\eta |Tu|)^{\gamma}g(\eta |Tu|)|Tu|\Big) X_lu X_i\eta
 \end{align*}
  and from \eqref{eq:comm}, recall that $X_{n+l}X_l = X_lX_{n+l}-T$ . Using these, we obtain
 \begin{equation*}
\begin{aligned}
\sum_{i=1}^{2n}\int_{\Omega}&\eta^{2} G(\eta |Tu|)^{\gamma+1}
X_l(\A_i(\X u))X_lX_iu\,dx\\
&=  \int_\Omega \eta^{2} G(\eta |Tu|)^{\gamma+1} X_l(\A_{n+l}(\X u)) Tu\dx\\
&\quad - (\gamma+1)\sum_{i=1}^{2n}\int_\Omega \eta^3 G(\eta |Tu|)^{\gamma}g(\eta |Tu|)X_luX_l(\A_i(\X u))X_i(|Tu|)\, dx\\
&\quad - \sum_{i=1}^{2n}\int_{\Omega}\Big(2\eta G(\eta|Tu|)+(\gamma+1)\eta^2 g(\eta|Tu|)|Tu|\Big) G(\eta|Tu|)^{\gamma}  X_lu
X_l(\A_i(\X u))X_i\eta\, dx\\
&\quad + \int_{\Omega} \eta^{2}G(\eta |Tu|)^{\gamma+1}  X_lu\, T(\A_{n+l}(\X u)) \, dx\\
& = I_1+I_2+I_3+I_4.
\end{aligned}
\end{equation*}
We shall estimate both sides of the above. To estimate the left hand side, we use 
the structure condition 
\eqref{eq:str cond diff}, to obtain 
\begin{equation*}
\sum_{i=1}^{2n}\int_{\Omega}\eta^{2} G(\eta |Tu|)^{\gamma+1}
X_l(\A_i(\X u))X_lX_iu\,dx
\geq \int_\Omega \eta^{2} G(\eta |Tu|)^{\gamma+1} \weight|X_l(\X u)|^2\, dx. 
\end{equation*}
For the right hand side, we claim the following 
for every $k\in \{1,2,3,4\}$,
\begin{equation}\label{eq:claim2}
 \begin{aligned}
 |I_k|\leq c\tau \int_\Omega \eta^{2} &G(\eta |Tu|)^{\gamma+1} \weight|\XX u|^2\, dx\\
 &+\frac{c}{\tau}\,(\gamma+1)^2 \|\X\eta\|_{L^\infty}^2\int_\Omega G(\eta |Tu|)^{\gamma+1} |Tu|^{-2}
|\X u|^2\weight  | \XX u|^2\, dx
 \end{aligned}
\end{equation}
for some $c = c(n,g_0,L)>0$, where $\tau>0$ is any arbitrary constant. Assuming 
the claim and 
combining it with the previous estimate, we end up with 
\begin{align*}
 \int_\Omega \eta^{2} G(\eta |Tu|)^{\gamma+1} &\weight|X_l(\X u)|^2\, dx
 \leq \tau \int_\Omega \eta^{2} G(\eta |Tu|)^{\gamma+1} 
 \weight|\XX u|^2\, dx\\
 &+\frac{c}{\tau}\,(\gamma+1)^2 \|\X\eta\|_{L^\infty}^2\int_\Omega G(\eta |Tu|)^{\gamma+1} |Tu|^{-2}
|\X u|^2\weight  | \XX u|^2\, dx
\end{align*}
for some $c = c(n,g_0,L)>0$ and every $l\in\{1,\ldots,n\}$. Similarly, the above inequality can also be obtained when 
$l\in\{n,\ldots,2n\}$. Then, by summing over the two inequalities and choosing 
$\tau>0$ small enough, it is easy to obtain \eqref{eq:Rev}, as required to 
complete the proof.

Thus, we are left with proving the claim \eqref{eq:claim2}, which we 
accomplish by estimating each $I_k$, one by one.  
For $I_1$, first we use integral by parts to get 
\begin{equation*}
\begin{aligned}
 I_1 &= \ -\int_\Omega X_l\Big(\eta^2G(\eta |Tu|)^{\gamma+1}  Tu\Big) \A_{n+l}(\X u)\,dx \\
   &=-\int_\Omega \eta^2G(\eta|Tu|)^\gamma \Big[ G(\eta|Tu|) + (\gamma+1)\eta |Tu|g(\eta|Tu|)\Big] \A_{n+l}(\X u)X_l(Tu)\,dx\\
   &\qquad\ - \int_\Omega \eta\, G(\eta|Tu|)^\gamma \Big[2G(\eta|Tu|)+(\gamma+1)\eta|Tu|g(\eta|Tu|)\Big] Tu\, \A_{n+l}(\X u)X_l\eta\,dx \\
   &= I_{11} + I_{12}. 
\end{aligned}
\end{equation*}
Recall that $ tg(t) \leq (1+g_0)G(t)$ for all $t>0$ from \eqref{eq:gG2}. Using this 
along with the structure condition \eqref{eq:str cond diff}, we will show that the claim \eqref{eq:claim2} holds for both $I_{11}$ and $ I_{12}$. 

For $I_{11}$, using \eqref{eq:gG2},\eqref{eq:str cond diff} and 
Young's inequality, we obtain
\begin{equation}\label{eq:I11}
\begin{aligned}
|I_{11}|\leq\ &c\,(\gamma+1)\int_\Omega \eta^2G(\eta |Tu|)^{\gamma+1} |\X u|\weight|\X(Tu)|\,dx\\
\leq\ &\frac{\tau}{\|\X\eta\|^2_{L^\infty}}\int_\Omega \eta^{4} G(\eta |Tu|)^{\gamma+1} \weight|\X(Tu)|^2\, dx\\
 &+ \frac{c}{\tau}\,(\gamma+1)^2\|\X\eta\|^2_{L^\infty}\int_\Omega  G(\eta |Tu|)^{\gamma+1} |\X u|^2\weight\, dx\\
\end{aligned}
\end{equation}
Now, the following inequality can be proved in a way similar to that of the Caccioppoli type inequality of $Tu$ in Lemma \ref{lem:cacci T}, with minor 
modifications, 
\begin{equation*}
\begin{aligned}
\int_\Omega \eta^4G(\eta |Tu|)^{\gamma+1} \weight |\X(Tu)|^2\,dx 
\leq c\int_\Omega \eta^2G(\eta |Tu|)^{\gamma+1} \weight| Tu|^2|\X\eta|^2\, dx,
\end{aligned}
\end{equation*}
for some $c=c(n,g_0,L)>0$. After using the above inequality for the first term of 
\eqref{eq:I11} and then using $|Tu|\leq 2|\XX u|$ for both terms, it is easy to see that \eqref{eq:claim2}
holds for $I_{11}$. 

For $I_{12}$, using structure condition \eqref{eq:str cond diff} and 
\eqref{eq:gG2} again, we get 
\begin{equation}\label{eq:I12}
|I_{12}|\leq c\,(\gamma+1)\int_\Omega |\eta|G(\eta |Tu|)^{\gamma+1} |\X u|
   \weight|Tu||\X\eta|\,dx
\end{equation}
from which, \eqref{eq:claim2} follows easily from Young's inequality and 
$|Tu|\leq 2|\XX u|$. Thus, combining the estimates \eqref{eq:I11} and 
\eqref{eq:I12}, we conclude that the claim \eqref{eq:claim2}, holds for $I_1$.

The estimate of $I_2$ is similar. We use 
 \eqref{eq:str cond diff}, 
\eqref{eq:gG2} and 
Young's inequality, to get
\begin{equation*}
\begin{aligned}
 |I_2| &\leq c\,(\gamma+1)\int_\Omega \eta^2G(\eta |Tu|)^{\gamma+1} 
 |Tu|^{-1}|\X u|\weight|\XX u||\X(Tu)|\,dx\\
 &\leq \frac{\tau}{\|\X\eta\|^2_{L^\infty}}\,\int_\Omega \eta^{4} 
 G(\eta |Tu|)^{\gamma+1} \weight|\X (Tu)|^2\, dx\\
 &\qquad + \frac{c}{\tau}\,(\gamma+1)^2\|\X\eta\|^2_{L^\infty}\int_\Omega  
 G(\eta |Tu|)^{\gamma+1} 
 |Tu|^{-2}|\X u|^2\weight|\XX u|^2\, dx.
\end{aligned}
\end{equation*}
Notice that, the first term on the right hand side of the latter inequality of the above, is identical to that of \eqref{eq:I11}. Hence, the claim 
\eqref{eq:claim2} for $I_2$, follows similarly. 

For $I_3$, using \eqref{eq:gG2} and structure condition \eqref{eq:str cond diff} again, we obtain
\begin{equation*}
\begin{aligned}
 |I_3| \leq  c\,(\gamma+1)\int_\Omega |\eta|\,G(\eta |Tu|)^{\gamma+1} 
 |\X u|\weight|\XX u||\X\eta|\,dx
\end{aligned}
\end{equation*}
which together with Young's inequality, is enough for claim \eqref{eq:claim2}.
Finally, the 
fourth term has the following estimate.
\begin{equation*}
\begin{aligned}
 |I_4| = \Big|\int_\Omega \eta^{2} &G(\eta |Tu|)^{\gamma+1}  X_lu
\sum_{i=1}^{2n}D_i \A_{n+l}(\X u)X_i(Tu)\,dx\Big|\\
 \leq &\int_\Omega \eta^2G(\eta |Tu|)^{\gamma+1} |\X u|\weight|\X(Tu)|\,dx,
\end{aligned}
\end{equation*}
which is identical to the upper bound of $I_{11}$ in \eqref{eq:I11}. Hence, 
the claim \eqref{eq:claim2} holds for $I_4$ as well and the proof is complete.
\end{proof}
The inequality \eqref{eq:Rev} of the above lemma yields the 
following intermediate inequality, which shall be essential for proving the 
final estimate for the horizontal gradient. 
\begin{Cor}\label{cor:Rev'}
 For any $\gamma\geq 1$ and all non-negative $\eta\in C^\infty_0(\Omega)$, we have
\begin{equation}\label{eq:Rev'}
\begin{aligned}
\int_\Omega \eta^{2} &G(\eta |Tu|)^{\gamma+1} \weight|\XX u|^2\, dx \\
&\leq c^\frac{\gamma+1}{2}(\gamma+1)^{(\gamma+1)(1+g_0)}\int_\Omega 
\eta^2 G\big(\|\X\eta\|_{L^\infty}|\X u|\big)^{\gamma+1}\weight|\XX u|^2 \dx
\end{aligned}
\end{equation}
where $c = c(n,g_0,L)>0$. 
\end{Cor}
\begin{proof}
Let us denote $\Psi(s) = \tau G(\sqrt{s})^{\gamma+1}$, where $\tau>0$ is an 
arbitrary constant. Notice that $\Psi$ is a N-function if $\gamma\geq 1$. Now 
we restate the inequality \eqref{eq:Rev} of Lemma \ref{lem:Rev}, as 
\begin{equation}\label{eq:Rev1}
\begin{aligned}
\int_\Omega \eta^{2} G(\eta |Tu|&)^{\gamma+1} \weight|\XX u|^2\dx\\
\leq \frac{c}{\tau}\,&(\gamma+1)^2 
\|\X\eta\|_{L^\infty}^2\int_\Omega 
\frac{\Psi(\eta^2|Tu|^2)}{|Tu|^2}
|\X u|^2\weight  | \XX u|^2\dx.
\end{aligned}
\end{equation}
Taking $\Psi^*$ as the conjugate function of $\Psi$, we apply the Young's inequality \eqref{eq:YI} on the right hand side 
of the above to get 
\begin{equation}\label{eq:YIapp}
\begin{aligned}
\frac{c}{\tau}\,(\gamma+1)^2 
&\|\X\eta\|_{L^\infty}^2\int_\Omega 
\frac{\Psi(\eta^2|Tu|^2)}{|Tu|^2}
|\X u|^2\weight  | \XX u|^2\dx\\
 &\leq \int_\Omega \eta^2\,
\Psi^*\bigg(\frac{\Psi(\eta^2|Tu|^2)}{\eta^2|Tu|^2}\bigg) 
\weight|\XX u|^2 \dx \\
&\qquad\quad +\int_\Omega \eta^2\,\Psi\Big(\frac{c}{\tau}(\gamma+1)^2 \|\X\eta\|_{L^\infty}^2|\X u|^2\Big)\weight|\XX u|^2 \dx .
\end{aligned}
\end{equation}
Recalling \eqref{eq:comp prop}, notice that  
$$ \Psi^*\bigg(\frac{\Psi(\eta^2|Tu|^2)}{\eta^2|Tu|^2}\bigg) 
\leq \Psi(\eta^2|Tu|^2) = \tau G(\eta|Tu|)^{\gamma+1} $$
and using this together with \eqref{eq:Rev1} and \eqref{eq:YIapp}, we end up with
 \begin{equation*}
\begin{aligned}
\int_\Omega \eta^{2} G(\eta &|Tu|)^{\gamma+1} \weight|\XX u|^2\dx\\
&\leq \tau\int_\Omega \eta^2\,G(\eta|Tu|)^{\gamma+1} 
\weight|\XX u|^2 \dx \\
&\qquad\quad +\int_\Omega \eta^2\,\tau G\Big(\sqrt{c/\tau}\,(\gamma+1) 
\|\X\eta\|_{L^\infty}|\X u|\Big)^{\gamma+1}\weight|\XX u|^2 \dx .
\end{aligned}
\end{equation*}
Thus, with a small enough 
$\tau>0$ and the doubling property of $G$, the proof is finished. 
\end{proof}
The inequality \eqref{eq:Rev'} is required in a slightly different form, which we state here in the following corollary. It is an easy consequence of Corollary \ref{cor:Rev'}, above.
\begin{Cor}\label{cor:Rev''}
For any $\gamma,\om\geq 1$ and all non-negative $\eta\in C^\infty_0(\Omega)$, 
we have 
\begin{equation}\label{eq:Rev''}
\begin{aligned}
\int_\Omega \eta^{2} &G\bigg(\frac{\eta|Tu|}{\sqrt{\om K_\eta}}\bigg)^{\gamma+1} \weight |\XX u|^2\dx\\
&\leq \frac{c^\frac{\gamma+1}{2}(\gamma+1)^{(\gamma+1)(1+g_0)}}{\om^{\frac{\gamma+1}{2}}}
\int_\Omega 
\eta^2 G\big(|\X u|\big)^{\gamma+1}\weight|\XX u|^2 \dx
\end{aligned}
\end{equation}
where $ K_\eta = \| \X \eta\|_{L^\infty(\Omega)}^2+\|\eta T\eta\|_{L^\infty(\Omega)} $ and 
$ c = c(n,g_0,L) >0$ is a constant.
\end{Cor}
\begin{proof}
Given any $\om\geq 1$, note that from Lemma \ref{lem:gandG prop},
\begin{equation}\label{eq:inqom}
G\bigg(\frac{t}{\sqrt{\om}}\bigg) \leq  \frac{t}{\sqrt{\om}}g\bigg(\frac{t}{\sqrt{\om}}\bigg) \leq \frac{1+g_0}{\sqrt{\om}} G(t) .
\end{equation}
  Taking $ K_\eta = \| \X \eta\|_{L^\infty(\Omega)}^2+\|\eta T\eta\|_{L^\infty(\Omega)} $, we use $\eta/\sqrt{\om K_\eta}$ in place of $\eta$ in \eqref{eq:Rev'}, to get that 
\begin{equation*}
\begin{aligned}
\int_\Omega \frac{\eta^{2}}{\om K_\eta} &G\bigg(\frac{\eta|Tu|}{\sqrt{\om K_\eta}}\bigg)^{\gamma+1}
 \weight|\XX u|^2\, dx \\
&\leq c^\frac{\gamma+1}{2}(\gamma+1)^{(\gamma+1)(1+g_0)}\int_\Omega 
\frac{\eta^{2}}{\om K_\eta} G
\bigg(\frac{\|\X\eta\|_{L^\infty}|\X u|}{\sqrt {\om K_\eta}}\bigg)^{\gamma+1}\weight|\XX u|^2 \dx\\
&\leq \frac{c^\frac{\gamma+1}{2}(1+g_0)^{\gamma+1}(\gamma+1)^{(\gamma+1)(1+g_0)}}{\om^{\frac{\gamma+1}{2}}}\int_\Omega 
\frac{\eta^{2}}{\om K_\eta} G(|\X u|)^{\gamma+1}\weight|\XX u|^2 \dx.
\end{aligned}
\end{equation*} 
In the latter inequality of the above, we have used  
$\|\X\eta\|_{L^\infty} \leq  \sqrt{K_\eta}$, monotonicity of $G$ and the inequality 
\eqref{eq:inqom}. 
After removing the factor $1/ \om K_\eta$ from both sides of the above, we end up with \eqref{eq:Rev''} for some $c=c(n,g_0,L)>0$, to complete the proof. 
\end{proof}

\subsection{Horizontal and Vertical estimates}\label{subsec:Horizontal and Vertical estimates}\noindent
\\
We first show that, the Caccioppoli type 
inequality of Lemma \ref{lem:cacci Xu}, can be improved using Corollary \ref{cor:Rev''}. This would be essential for the proof of Theorem \ref{thm:mainthm}.
\begin{Prop}\label{prop:horizontal estimate} 
 If $ u\in HW^{1,G}(\Om) $ is a weak solution of equation \eqref{eq:maineq}, then 
 for any $\gamma\geq 1$ and all non-negative $\eta\in C^\infty_0(\Omega)$, we have 
 the following estimate
 \begin{equation*}
  \int_{\Omega}\eta^2G(|\X u|)^{\gamma+1}\weight|\XX u|^2\, dx\leq  
 c(\gamma+1)^{10(1+g_0)}K_\eta\int_{\supp(\eta)} G(|\X u|)^{\gamma+1}|\X u|^2\weight \dx,
 \end{equation*}
where $ K_\eta = \| \X \eta\|_{L^\infty(\Omega)}^2+\|\eta T\eta\|_{L^\infty(\Omega)} $ and 
$ c = c(n,g_0,L) >0$ is a constant.
\end{Prop}

\begin{proof}
First, we recall the Caccioppoli type estimate of Lemma \ref{lem:cacci Xu}, 
\begin{equation}\label{eq:cacci Xu}
\begin{aligned}
\int_{\Om}\eta^2G(|\X u|)^{\gamma+1}\weight|\XX u|^2 \dx&\leq c K_\eta\int_\Om G(|\X u|)^{\gamma+1}|\X u|^2\weight dx\\
&\quad+ c\,(\gamma+1)^4\int_\Om\eta^2G(|\X u|)^{\gamma+1}\weight| Tu|^2\dx,
\end{aligned}
\end{equation}
where $ K_\eta = \| \X \eta\|_{L^\infty(\Omega)}^2+\|\eta T\eta\|_{L^\infty(\Omega)} $ and $ c = c(n,g_0,L) >0$. Thus, to complete the proof, 
we require an estimate of the second integral of the right hand side of the above.

To this end, let us denote 
\begin{equation}\label{eq:Phi}
\Phi(s) = \om K_\eta \,s\,G(\sqrt{s})^{\gamma+1}
\end{equation}
where $\om \geq 1$ is a constant which shall be specified later. Let $\Phi^*$ be the conjugate of $\Phi$. We estimate the last integral of 
\eqref{eq:cacci Xu} using the Young's inequality \eqref{eq:YI}, as follows;
\begin{equation*}
\begin{aligned}
c(\gamma+1)^4&\int_\Omega \eta^2 G(|\X u|)^{\gamma+1}\weight| Tu|^2\dx\\
&\leq \int_\Omega \Phi\bigg(c(\gamma+1)^4\frac{\eta^2|Tu|^2}{\om K_\eta}\bigg)\weight \dx+ \int_\Omega \Phi^*\Big( \om K_\eta G(|\X u|)^{\gamma+1}\Big)\weight\dx\\
&= Z_1 + Z_2
\end{aligned}
\end{equation*}
where $Z_1$ and $Z_2$ are the respective terms of the right hand side. 
Now, we estimate $Z_1$ and $Z_2$, one by one. 
First, using \eqref{eq:Phi}, doubling property for $G$ and $|Tu|\leq 2|\XX u|$, notice that
\begin{equation}\label{eq:Z1f}
\begin{aligned}
Z_1 &= c(\gamma+1)^4\int_\Omega \eta^{2} |Tu|^2G\bigg(\sqrt{c}(\gamma+1)^2\frac{\eta|Tu|}{\sqrt{\om K_\eta}}\bigg)^{\gamma+1}\weight \dx\\
&\leq c^\frac{\gamma+1}{2}(\gamma+1)^{4+2(\gamma+1)(1+g_0)}\int_\Omega \eta^{2} G\bigg(\frac{\eta|Tu|}{\sqrt{\om K_\eta}}\bigg)^{\gamma+1}
 \weight|\XX u|^2\, dx
\end{aligned}
\end{equation}
for some $c = c(n,g_0,L)>0$. 
Now, we apply the estimate \eqref{eq:Rev''} from Corrollary \ref{cor:Rev''} on  the last term of \eqref{eq:Z1f}, to get that 
\begin{equation}\label{eq:Z1}
\begin{aligned}
Z_1 &\leq 
\frac{c^{\gamma+1}(\gamma+1)^{4+3(\gamma+1)(1+g_0)}}{\om^{\frac{\gamma+1}{2}}} 
 \int_\Omega 
\eta^2 G\big(|\X u|\big)^{\gamma+1}\weight|\XX u|^2 \dx \\
&= \frac{1}{2}\int_\Omega 
\eta^2 G\big(|\X u|\big)^{\gamma+1}\weight|\XX u|^2 \dx ,
\end{aligned}
\end{equation}
where $\om$ is chosen as 
\begin{equation}\label{eq:om}
\om = 2^\frac{2}{\gamma+1}
c^2(\gamma+1)^{6(1+g_0)+\frac{8}{\gamma+1}}, 
\end{equation}
where $c$ is the constant $c = c(n,g_0,L)>0$ in the first step of \eqref{eq:Z1}. 

To estimate $Z_2$, first notice that, from the inequality \eqref{eq:comp prop} 
and the definition \eqref{eq:Phi}
\begin{equation}\label{eq:ob1}
\Phi^*\Big( \om K_\eta G(|\X u|)^{\gamma+1}\Big) = 
\Phi^*\bigg(\frac{\Phi(|\X u|^2)}{|\X u|^2}\bigg) 
\leq \Phi(|\X u|^2) =  \om K_\eta |\X u|^2\,G(|\X u|)^{\gamma+1}.
\end{equation}
Using the above, we immediately have that 
\begin{equation}\label{eq:Z2}
Z_2 
\leq \om K_\eta\int_\Om G(|\X u|)^{\gamma+1}|\X u|^2\weight dx.
\end{equation}
Combining \eqref{eq:Z1} and \eqref{eq:Z2} with $\om$ as in \eqref{eq:om}, we finally 
end up with 
\begin{equation*}
\begin{aligned}
c(\gamma+1)^4\int_\Omega &\eta^2 G(|\X u|)^{\gamma+1}\weight| Tu|^2\dx\leq \frac{1}{2}\int_\Omega 
\eta^2 G\big(|\X u|\big)^{\gamma+1}\weight|\XX u|^2 \dx\\
&\qquad + c(\gamma+1)^{6(1+g_0)+\frac{8}{\gamma+1}}K_\eta\int_\Om G(|\X u|)^{\gamma+1}|\X u|^2\weight dx
\end{aligned}
\end{equation*}
for some $c = c(n,g_0,L)>0$. This, together with \eqref{eq:cacci Xu}, is enough 
to conclude the proof. 
\end{proof}

The following local estimate for the vertical derivative is an immediate consequence of 
the horizontal estimate of Proposition \ref{prop:horizontal estimate} and Corrollary \ref{cor:Rev''}, 
with the use of $ |Tu| \leq 2|\XX u| $.

\begin{Cor}\label{cor:vertical estimate}
 If $ u\in HW^{1,G}(\Om) $ is a weak solution of equation \eqref{eq:maineq}, then
 for any $\gamma\geq 1$ and all non-negative $\eta\in C^\infty_0(\Omega)$, we have
 the following estimate.
 \begin{equation*}
  \begin{aligned}
\int_\Omega \eta^{2} G\bigg(\frac{\eta|Tu|}{\sqrt{K_\eta}}\bigg)^{\gamma+1} \weight |T u|^2\dx
   \leq c(\gamma)K_\eta \int_{\supp(\eta)} G(|\X u|)^{\gamma+1}|\X u|^2\weight \dx
  \end{aligned}
 \end{equation*} 
where $ K_\eta = \| \X \eta\|_{L^\infty(\Omega)}^2+\|\eta T\eta\|_{L^\infty(\Omega)} $ and 
$ c(\gamma) = c(n,g_0,L,\gamma) >0$ is a constant.
\end{Cor}
\subsection{Proof of Theorem \ref{thm:mainthm}}\label{subsec:Proof of theorem}\noindent
\\
We recall that all the estimates above, rely on the apriori assumptions \eqref{eq:ass1} and \eqref{eq:ass2}. We prove Theorem \ref{thm:mainthm} here in three steps; 
first by assuming both \eqref{eq:ass1} and 
\eqref{eq:ass2}, then by removing them one by one.
\begin{proof}[Proof of Theorem \ref{thm:mainthm}]
First note that, it is enough to establish the estimate \eqref{eq:locbound} to finish the proof. If \eqref{eq:locbound} holds apriori for a weak solution $ u \in HW^{1,G}(\Omega) $ of \eqref{eq:maineq}, then 
monotonicity of $g$  
immediately implies $|\X u| \in L^\infty(B_{\sigma r}) $ along with the estimate 
 $$  \sup_{B_{\sigma r}}|\X u|
 \leq \max\left\{\ 1\,,\, 
 \frac{c(n,g_0,\delta,L)}{g(1)(1-\sigma)^Q}\vint_{B_r} G(|\X u|)\dx \ \right\}.$$
 {\it Step 1 : We assume both \eqref{eq:ass1} and \eqref{eq:ass2}.}
 
 The estimate \eqref{eq:locbound} follows from Proposition \ref{prop:horizontal estimate} by standard Moser's iteration. Here, we provide a brief outline. 
 Letting $ w = G(|\X u|) $, note that from \eqref{eq:gG2} 
 $$ |\X w|^2 \leq  
 |\X u|^2\weight^2 |\XX u|^2 \leq (1+g_0) w\, \weight |\XX u|^2,$$ 
and hence, from Proposition \ref{prop:horizontal estimate} we obtain
\begin{equation}\label{eq:moser}
\int_\Omega \eta^2 w^\gamma  |\X w |^2\,dx \leq  
 c\,(\gamma+1)^{10(1+g_0)}\, K_\eta\int_{\supp(\eta)} w^{\gamma+2}\,dx 
\end{equation} 
for some $ c = c(n,g_0,L) >0$ and $ K_\eta = 
\| \X \eta\|_{L^\infty(\Omega)}^2+\|\eta T\eta\|_{L^\infty(\Omega)} $. 
Now we use a standard choice of test function $ \eta \in C^\infty_0(B_r) $ such that 
$ 0\leq \eta \leq 1$ and $\eta \equiv 1$ in $B_{r'}$ for $ 0<r'< r$, 
$$
|\X \eta|\leq 4/(r-r') \quad \text{and}\quad |\XX \eta| \leq 16n/(r-r')^2.$$ 
Letting   
$ \kappa = Q/(Q-2)$ and using Sobolev's inequality \eqref{eq:sob emb} for $ q =2$
on \eqref{eq:moser}, we get that
$$ \left(\int_{B_{r'}}w^{(\gamma+2)\kappa}\,dx \right)^\frac{1}{\kappa}\leq  
\ \frac{c(\gamma+2)^{12(1+g_0)}}{(r-r')^2}\int_{B_r} w^{\gamma+2}\,dx $$ 
for every $\gamma\geq 1$. Iterating this with $\gamma_i = 3\kappa^i-2$ and 
$ r_i = \sigma r + (1-\sigma)r/2^i $, 
we get 
$$  \sup_{B_{\sigma r}}\ w \leq \frac{c}{(1-\sigma)^{Q/3}}\bigg(\intav_{B_r}
  w^3\,dx\bigg)^\frac{1}{3}$$
 for $c = c(n,g_0,L)>0$ and this holds for every $B_r\subset\Om$ and every $0<\sigma<1$. 
Then, a standard interpolation argument (see \cite{Dib-Tru}, p. 299--300) leads to 
$$  \sup_{B_{\sigma r}}\ w \leq \frac{c(q)}{(1-\sigma)^{Q/q}}\bigg(\intav_{B_r}
  w^q\,dx\bigg)^\frac{1}{q} $$
for every $q>0$ and some $c(q)=c(n,g_0,L,q)>0$. Taking $q=1$, we get the estimate \eqref{eq:locbound}.\noindent
\\\\
{\it Step 2 : We assume \eqref{eq:ass1} but remove \eqref{eq:ass2}.}

Let $B_r = B_r(x_0) \subset \Om$ be a fixed CC-ball. 
Given the weak solution $ u \in HW^{1,G}(\Omega) $, there exists 
a smooth approximation $\phi_m \in C^\infty(B_r)$ such that $\phi_m \to u$ in 
$HW^{1,G}(B_r)$ as $m\to\infty$. By virtue of equivalence with the  
Kor\`anyi metric, it is possible to find a concentric ball $K_{\theta r}\subset\subset B_r$ with respect to the norm \eqref{eq:norm}, 
for some constant $\theta = \theta(n)>0$.

Now, let $u_m$ be the weak solution of the following Dirichlet problem,
\begin{equation}\label{eq:ADprob}
\begin{cases}
&\dvh (\A(\X u_m))= \ 0 \quad  \text{ in } K_{\theta r}\\
& u_m-\phi_m\in HW^{1,G}_0(K_{\theta r}).
\end{cases}
\end{equation}
The choice of test function $u_m-\phi_m$ on \eqref{eq:ADprob}, yields
\begin{equation}\label{eq:el1}
 \int_{K_{\theta r}}\inp{\A(\X u_m)}{\X u_m}\dx 
=  \int_{K_{\theta r}}\inp{\A(\X u_m)}{\X \phi_m}\dx
\end{equation}
Now, there exists $k=c(g_0,L)>1$ such that combining ellipticity \eqref{eq:elliptic} and structure condition \eqref{eq:str cond diff}, one has 
$\inp{\A(z)}{z} \geq (2/k)\,|z||\A(z)| $. 
Using this along with \eqref{eq:str cond diff} and doubling property of $g$, we estimate the right hand side of 
\eqref{eq:el1}, as 
\begin{equation*}
  \begin{aligned}
   \int_{K_{\theta r}}\inp{\A(\X u_m)}{\X \phi_m}\dx 
   & = \int_{|\X u_m|\geq k|\X \phi_m|}\inp{\A(\X u_m)}{\X \phi_m}\dx 
   + \int_{|\X u_m|< k|\X \phi_m|}\inp{\A(\X u_m)}{\X \phi_m}\dx\\
   &\leq \frac{1}{k}\int_{K_{\theta r}} |\A(\X u_m)||\X u_m|\dx 
   \,+\, \int_{|\X u_m|< k|\X \phi_m|}L \,g(|\X u_m|)\,|\X \phi_m|\dx\\
   &\leq \frac{1}{2}\int_{K_{\theta r}}\inp{\A(\X u_m)}{\X u_m}\dx 
   \,+\,  k^{g_0}L \int_{K_{\theta r}}g(|\X \phi_m|)\,|\X \phi_m|\dx.
  \end{aligned}
 \end{equation*}
Combining the above with \eqref{eq:el1} and using \eqref{eq:elliptic}, we get 
\begin{equation}\label{eq:el2}
 \int_{K_{\theta r}} G(|\X u_m|)\dx\leq c\int_{K_{\theta r}} G(|\X \phi_m|)\dx
\leq c\int_{K_{\theta r}} G(|\X u|)\dx+ o(1/m)
\end{equation}
for $c = c(n,g_0,L) >0$ and $o(1/m) \to 0$ as $m\to\infty$. Now,  
since $\phi_m$ is smooth and $K_{\theta r}$ 
(defined by norm \eqref{eq:norm}) satisfies the strong convexity condition 
\eqref{eq:str conv}, the equation \eqref{eq:ADprob} is an example of the Dirichlet problem 
\eqref{eq:Dprob}. From Proposition \ref{prop:lip existence}, we have that 
$$ \| \X u_m\|_{L^\infty(K_{\theta r})} \leq M $$ 
which is the assumption \eqref{eq:ass2} for $u_m$. Now we can apply Step 1 and 
conclude 
\begin{equation}\label{eq:epsG}
 \sup_{B_{\sigma \tau r}}\ G(|\X u_m|)
 \leq \frac{c}{(1-\sigma)^Q}\intav_{B_{\tau r}}G(|\X u_m|)\dx
\end{equation}
for some $c = c(n,g_0,L) >0,\ \sigma\in(0,1)$ and $\tau = \tau(n)>0$ chosen 
such that $B_{\tau r} \subset K_{\theta r}$. 
This is followed up with standard argument, since \eqref{eq:el2} ensures 
that there exists $\tilde u \in HW^{1,G}(K_{\theta r})$ such that 
upto a subsequence $u_m \wto \tilde u$. 
Since, $u_m-\phi_m \in  HW^{1,G}_0(K_{\theta r})$, hence we have 
$\tilde u-u\in HW^{1,G}_0(K_{\theta r})$ and 
combined with the monotonicity \eqref{eq:monotone}, one can show $\tilde u$ is a weak solution of \eqref{eq:maineq}. From uniqueness, 
$\tilde u = u$. Taking $m\to\infty$ in 
\eqref{eq:epsG} and \eqref{eq:el2}, we conclude
$$ \sup_{B_{\sigma \tau r}}\ G(|\X u|)
 \leq \frac{c}{(1-\sigma)^Q}\intav_{B_r}G(|\X u|)\dx $$
and \eqref{eq:locbound} follows from a simple covering argument. \noindent
\\\\
{\it Step 3: We remove both \eqref{eq:ass2} and \eqref{eq:ass1}.}

The assumption \eqref{eq:ass1} is removed by a standard approximation argument.  We use the regularization constructed in Lemma 5.2 of \cite{Lieb--gen}. Here, we 
give a brief outline.

For any fixed $ 0< \eps < 1$ and some $ \eta_\eps \in C^{0,1}([0,\infty))$, we  define 
\begin{equation}\label{eq:regularization}
\F_\eps(t)=  \F\Big(\min\{\,t+\eps\,,\,1/\eps\,\}\Big) 
\quad\text{and}\quad 
\Aeps(z) =  \eta_\eps(|z|)\F_\eps(|z|)\,z + \Big(1-\eta_\eps(|z|)\Big)\A(z)
\end{equation}
where $\A$ is given and $\F(t) = g(t)/t$. Thus, $\F_\eps$ satisfies the assumption \eqref{eq:ass1} with 
$m_1 = \F(\eps)$ and $m_2=\F(1/\eps)$. Also, with the choice of 
$\eta_\eps$ as in \cite{Lieb--gen}(p. 343), it is possible to show that
\begin{equation}\label{eq:str cond reg}
\begin{aligned}
\frac{1}{\tilde L}\F_\eps(|z|)|\xi|^2 \leq \,&\inp{D\Aeps(z)\,\xi}{\xi}
\leq \tilde L\,\F_\eps(|z|)|\xi|^2;\\
&|\Aeps(z)|\leq \tilde L |z|\F_\eps(|z|),
\end{aligned}
\end{equation}
for some $\tilde L= \tilde L(\delta,g_0,L)>0$. 
Reducing to a subsequence
if necessary, it is easy to see that $\Aeps \to \A$ uniformly and 
$\F_\eps \to \F$ uniformly on compact subsets of $(0,\infty)$, as $\eps \to 0$. 

Given weak solution $u\in HW^{1,G}(\Om)$ of \eqref{eq:maineq}, 
we consider $u_\eps$ as the weak solution of the following regularized equation 
\begin{equation}\label{eq:dirichlet prob reg}
 \begin{cases}
  -\dvh (\Aeps(\X u_\eps))=  0\ \ \text{in}\ \Om';\\
 \ \ \ u_\eps - u\in HW^{1,G}_0(\Om'), 
 \end{cases}
\end{equation}
for any $\Om'\subset\subset \Om$. Now, we are able to apply Step 2, to obtain 
uniform estimates for $u_\eps$. Taking limit 
$\eps \to 0$, we can obtain 
\eqref{eq:locbound}. This concludes the proof.
\end{proof}



\bibliographystyle{plain}
\bibliography{LipH}

\def\cprime{$'$} \def\cprime{$'$} \def\cprime{$'$}
\begin{thebibliography}{10}

\bibitem{Bonfig-Lanco-Ugu}
A.~Bonfiglioli, E.~Lanconelli, and F.~Uguzzoni.
\newblock {\em Stratified {L}ie groups and potential theory for their
  sub-{L}aplacians}.
\newblock Springer Monographs in Mathematics. Springer, Berlin, 2007.

\bibitem{Cap--reg}
Luca Capogna.
\newblock Regularity of quasi-linear equations in the {H}eisenberg group.
\newblock {\em Comm. Pure Appl. Math.}, 50(9):867--889, 1997.

\bibitem{C-D-G}
Luca Capogna, Donatella Danielli, and Nicola Garofalo.
\newblock An embedding theorem and the {H}arnack inequality for nonlinear
  subelliptic equations.
\newblock {\em Comm. Partial Differential Equations}, 18(9-10):1765--1794,
  1993.

\bibitem{C-D-S-T}
Luca Capogna, Donatella Danielli, Scott~D. Pauls, and Jeremy~T. Tyson.
\newblock {\em An introduction to the {H}eisenberg group and the
  sub-{R}iemannian isoperimetric problem}, volume 259 of {\em Progress in
  Mathematics}.
\newblock Birkh\"auser Verlag, Basel, 2007.

\bibitem{Cap-Garo}
Luca Capogna and Nicola Garofalo.
\newblock Regularity of minimizers of the calculus of variations in {C}arnot
  groups via hypoellipticity of systems of {H}\"ormander type.
\newblock {\em J. Eur. Math. Soc. (JEMS)}, 5(1):1--40, 2003.

\bibitem{Chow}
Wei-Liang Chow.
\newblock \"{U}ber {S}ysteme von linearen partiellen {D}ifferentialgleichungen
  erster {O}rdnung.
\newblock {\em Math. Ann.}, 117:98--105, 1939.

\bibitem{Dib}
E.~DiBenedetto.
\newblock {$C^{1+\alpha }$} local regularity of weak solutions of degenerate
  elliptic equations.
\newblock {\em Nonlinear Anal.}, 7(8):827--850, 1983.

\bibitem{Dib-Tru}
E.~DiBenedetto and Neil~S. Trudinger.
\newblock Harnack inequalities for quasiminima of variational integrals.
\newblock {\em Ann. Inst. H. Poincar\'e Anal. Non Lin\'eaire}, 1(4):295--308,
  1984.

\bibitem{Dom}
Andr{\'a}s Domokos.
\newblock Differentiability of solutions for the non-degenerate
  {$p$}-{L}aplacian in the {H}eisenberg group.
\newblock {\em J. Differential Equations}, 204(2):439--470, 2004.

\bibitem{Dom-Man--reg}
Andr{\'a}s Domokos and Juan~J. Manfredi.
\newblock {$C^{1,\alpha}$}-regularity for {$p$}-harmonic functions in the
  {H}eisenberg group for {$p$} near 2.
\newblock 370:17--23, 2005.

\bibitem{Dom-Man--cordes}
Andr{\'a}s Domokos and Juan~J. Manfredi.
\newblock Subelliptic {C}ordes estimates.
\newblock {\em Proc. Amer. Math. Soc.}, 133(4):1047--1056 (electronic), 2005.

\bibitem{Evans}
Lawrence~C. Evans.
\newblock A new proof of local {$C^{1,\alpha }$}\ regularity for solutions of
  certain degenerate elliptic p.d.e.
\newblock {\em J. Differential Equations}, 45(3):356--373, 1982.

\bibitem{Foglein}
Anna F{\"o}glein.
\newblock Partial regularity results for subelliptic systems in the
  {H}eisenberg group.
\newblock {\em Calc. Var. Partial Differential Equations}, 32(1):25--51, 2008.

\bibitem{Folland-Stein--book}
G.~B. Folland and Elias~M. Stein.
\newblock {\em Hardy spaces on homogeneous groups}, volume~28 of {\em
  Mathematical Notes}.
\newblock Princeton University Press, Princeton, N.J.; University of Tokyo
  Press, Tokyo, 1982.

\bibitem{Fusco-Sb}
Nicola Fusco and Carlo Sbordone.
\newblock Higher integrability of the gradient of minimizers of functionals
  with nonstandard growth conditions.
\newblock {\em Comm. Pure Appl. Math.}, 43(5):673--683, 1990.

\bibitem{Gia-Giu--div}
M.~Giaquinta and E.~Giusti.
\newblock Global {$C^{1,\alpha }$}-regularity for second order quasilinear
  elliptic equations in divergence form.
\newblock {\em J. Reine Angew. Math.}, 351:55--65, 1984.

\bibitem{Gia-Giu--min}
Mariano Giaquinta and Enrico Giusti.
\newblock On the regularity of the minima of variational integrals.
\newblock {\em Acta Math.}, 148:31--46, 1982.

\bibitem{Hein-Kilp-Mar}
Juha Heinonen, Tero Kilpel{\"a}inen, and Oll-i Martio.
\newblock {\em Nonlinear potential theory of degenerate elliptic equations}.
\newblock Dover Publications, Inc., Mineola, NY, 2006.
\newblock Unabridged republication of the 1993 original.

\bibitem{Hor}
Lars H{\"o}rmander.
\newblock Hypoelliptic second order differential equations.
\newblock {\em Acta Math.}, 119:147--171, 1967.

\bibitem{Kind-Stamp}
David Kinderlehrer and Guido Stampacchia.
\newblock {\em An introduction to variational inequalities and their
  applications}, volume~31 of {\em Classics in Applied Mathematics}.
\newblock Society for Industrial and Applied Mathematics (SIAM), Philadelphia,
  PA, 2000.
\newblock Reprint of the 1980 original.

\bibitem{Kuf-O-F}
Alois Kufner, Old{\v{r}}ich John, and Svatopluk Fu{\v{c}}{\'{\i}}k.
\newblock {\em Function spaces}.
\newblock Noordhoff International Publishing, Leyden; Academia, Prague, 1977.
\newblock Monographs and Textbooks on Mechanics of Solids and Fluids;
  Mechanics: Analysis.

\bibitem{Lady-Ural}
Olga~A. Ladyzhenskaya and Nina~N. Ural{\cprime}tseva.
\newblock {\em Linear and quasilinear elliptic equations}.
\newblock Translated from the Russian by Scripta Technica, Inc. Translation
  editor: Leon Ehrenpreis. Academic Press, New York-London, 1968.

\bibitem{Lewis}
John~L. Lewis.
\newblock Regularity of the derivatives of solutions to certain degenerate
  elliptic equations.
\newblock {\em Indiana Univ. Math. J.}, 32(6):849--858, 1983.

\bibitem{Lieb--gen}
Gary~M. Lieberman.
\newblock The natural generalization of the natural conditions of
  {L}adyzhenskaya and {U}ral\cprime tseva for elliptic equations.
\newblock {\em Comm. Partial Differential Equations}, 16(2-3):311--361, 1991.

\bibitem{Man-Min}
Juan~J. Manfredi and Giuseppe Mingione.
\newblock Regularity results for quasilinear elliptic equations in the
  {H}eisenberg group.
\newblock {\em Math. Ann.}, 339(3):485--544, 2007.

\bibitem{Mar1}
Paolo Marcellini.
\newblock Regularity and existence of solutions of elliptic equations with
  {$p,q$}-growth conditions.
\newblock {\em J. Differential Equations}, 90(1):1--30, 1991.

\bibitem{Mar2}
Paolo Marcellini.
\newblock Regularity for elliptic equations with general growth conditions.
\newblock {\em J. Differential Equations}, 105(2):296--333, 1993.

\bibitem{Marchi}
Silvana Marchi.
\newblock {$C^{1,\alpha}$} local regularity for the solutions of the
  {$p$}-{L}aplacian on the {H}eisenberg group. {T}he case {$1+\frac 1{\sqrt
  5}<p\leq 2$}.
\newblock {\em Comment. Math. Univ. Carolin.}, 44(1):33--56, 2003.

\bibitem{Min-Z-Zhong}
Giuseppe Mingione, Anna Zatorska-Goldstein, and Xiao Zhong.
\newblock Gradient regularity for elliptic equations in the {H}eisenberg group.
\newblock {\em Adv. Math.}, 222(1):62--129, 2009.

\bibitem{Muk}
Shirsho Mukherjee.
\newblock {$C^{1,\alpha}$-Regularity of Quasilinear equations on the
  {H}eisenberg {G}roup}.
\newblock {\em \url{https://arxiv.org/abs/1805.03748}}, 2018.

\bibitem{Muk-Zhong}
Shirsho {Mukherjee} and Xiao {Zhong}.
\newblock {$C^{1,\alpha}$-Regularity for variational problems in the
  {H}eisenberg {G}roup}.
\newblock {\em \url{https://arxiv.org/abs/1711.04671}}, 2017.

\bibitem{Rao-Ren}
M.~M. Rao and Z.~D. Ren.
\newblock {\em Theory of {O}rlicz spaces}, volume 146 of {\em Monographs and
  Textbooks in Pure and Applied Mathematics}.
\newblock Marcel Dekker, Inc., New York, 1991.

\bibitem{Simon}
Leon Simon.
\newblock Interior gradient bounds for non-uniformly elliptic equations.
\newblock {\em Indiana Univ. Math. J.}, 25(9):821--855, 1976.

\bibitem{Tolk}
Peter Tolksdorf.
\newblock Regularity for a more general class of quasilinear elliptic
  equations.
\newblock {\em J. Differential Equations}, 51(1):126--150, 1984.

\bibitem{Tuo}
Heli Tuominen.
\newblock Orlicz-{S}obolev spaces on metric measure spaces.
\newblock {\em Ann. Acad. Sci. Fenn. Math. Diss.}, 135:86, 2004.
\newblock Dissertation, University of Jyv{\"a}skyl{\"a}, Jyv{\"a}skyl{\"a},
  2004.

\bibitem{Uhlen}
K.~Uhlenbeck.
\newblock Regularity for a class of non-linear elliptic systems.
\newblock {\em Acta Math.}, 138(3-4):219--240, 1977.

\bibitem{Zhong}
Xiao Zhong.
\newblock Regularity for variational problems in the {H}eisenberg {G}roup.
\newblock {\em \url{https://arxiv.org/abs/1711.03284}}.

\end{thebibliography}

\end{document}